\documentclass[10pt]{article}
\usepackage{graphicx}
\usepackage{amsthm,amsfonts,amssymb,amsmath,oldgerm}
\usepackage{fullpage}
\usepackage{graphicx}
\usepackage{mathrsfs}
\usepackage{xcolor}
\numberwithin{equation}{section}
\usepackage{enumitem}
\usepackage{hyperref}
\usepackage{subcaption}

\makeatletter
\def\namedlabel#1#2{\begingroup
    #2
    \def\@currentlabel{#2}
    \phantomsection\label{#1}\endgroup
}
\makeatother

\newcommand{\R}{\mathbb R}
\newcommand{\C}{\mathbb C}

\renewcommand{\Re}{\mathrm{Re}}
\renewcommand{\Im}{\mathrm{Im}}

\newcommand{\w}{\mathbf{w}}
\renewcommand{\u}{\mathbf{u}}

\renewcommand{\v}{\mathbf{v}}

\newcommand{\El}{\mathcal{L}}

\newtheorem{theorem}{Theorem}[section]
\newtheorem{proposition}[theorem]{Proposition}
\newtheorem{corollary}[theorem]{Corollary}
\newtheorem{lemma}[theorem]{Lemma}
\newtheorem{remark}[theorem]{Remark}
\theoremstyle{definition}

\allowdisplaybreaks[3]

\title{Nonlinear dynamics of periodic Lugiato-Lefever waves against sums of co-periodic and localized perturbations}
\author{Joannis Alexopoulos\thanks{Department of Mathematics, Karlsruhe Institute of Technology, Englerstra\ss e 2, 76131 Karlsruhe, Germany;\\ \texttt{joannis.alexopoulos@kit.edu}}}

\begin{document}

\maketitle

\begin{abstract}
In recent years, essential progress has been made in the nonlinear stability analysis of periodic Lugiato-Lefever waves against co-periodic and localized perturbations. Inspired by considerations from fiber optics, we introduce a novel iteration scheme which allows to perturb against sums of co-periodic and localized functions. This unifies previous stability theories in a natural manner. 
\paragraph*{Keywords.} Nonlinear stability; Periodic waves; Lugiato-Lefever equation; Nonlinear Schr\"odinger equations; Nonlocalized perturbations; Tooth problem \\
\textbf{Mathematics Subject Classification (2020).} 35B10; 35B35; 35Q55
\end{abstract}

\section{Introduction}
We study the Lugiato-Lefever equation on the extended real line \begin{align}
\label{LLE}
    \partial_t u = -\beta i \partial_x^2 u - (1+i\alpha) u + i |u|^2 u + F, \quad \beta \in \{-1,1\}, \quad \alpha \in \R, \quad  F>0,
\end{align}
for $u:\R\times [0,\infty) \rightarrow \C$, which is a model from nonlinear optics \cite{chembo}. An important observation is that the  principal part $-\beta i \partial_x^2 u$ is dispersive while the damping term $-u$ causes energy dissipation. The forcing term $F$ again adds energy to the physical system and allows for pattern formation as predicted by Lugiato and Lefever in  \cite{original_lle}. The most fundamental patterns such as pulses, small-amplitude or dnoidal periodic waves are found in \cite{pulse1, pulse2, reichel_pulse, bengel} and in \cite{HD} and \cite{dnoidal}, respectively. Recently, in \cite{BengeldeRijk}, the authors have obtained large-amplitude periodic waves. 

In this paper,  we prove nonlinear $L^\infty$-stability of $T$-periodic standing waves against initial perturbations from the space $L^2_{\textrm{\textrm{per}}}(0,T) \oplus L^2(\R)$\footnote{ The term $\oplus$ denotes the direct sum, that is for every $u \in L^2_{\textrm{\textrm{per}}}(0,T) \oplus L^2(\R)$, we find precisely one $w \in L^2_{\textrm{\textrm{per}}}(0,T)$
and $v \in L^2(\R)$ such that $u = w + v$.} under diffusive spectral stability assumptions. These spectral assumptions are only established in \cite{BengeldeRijk} and \cite{HD}.

We emphasize that sums of periodic and localized perturbations are not necessarily localized or periodic and thus our result is a nontrivial unification of the theories \cite{LLE_periodic} (co-periodic) and \cite{haragus} (localized). 
For the precise formulation of our main result, we refer to \S\ref{section_main_result}. 

 In view of fully nonlocalized perturbations, the recently developed nonlinear stability theory for dissipative semilinear systems \cite{ours} is not immediately applicable to all pattern-forming semilinear systems such as the Lugiato-Lefever equation. From this perspective, the present paper is the first to accommodate nonlocalized perturbations and to combine $H^l_{\textrm{per}}(0,T)$- and $H^k$-theory.
 On the other hand, considerations from fiber optics, see \cite{application} and Remark \ref{remark_fiber}, where combinations of localized and co-periodic effects naturally occur, motivate the investigation of these types of perturbations. Interpretating the Lugiato-Lefever equation as a variant of the cubic nonlinear Schr\"odinger equation,
a related inspiration for this paper comes from the so-called \textit{tooth problem} asking whether solutions of the nonlinear Schr\"odinger equation with not necessarily small initial data from $L^2_{\textrm{\textrm{per}}}(\R) \oplus L^2(\R)$ exist globally; we refer to  \cite{kunst_klaus} and \cite{leonid_and_so_on} for answers to the tooth problem and background information.

Our central challenge is to develop the right modulational ansatz to make a nonlinear iteration argument close. Inspired by \cite{uniform}, we introduce both a temporal and a localized spatio-temporal phase modulation to capture the critical dynamics of the perturbation induced by translational invariance of (\ref{LLE}). Moreover, we employ novel $L^2$-$L^\infty$-estimates in the nonlinear iteration argument to control the interaction between the periodic and localized components of the perturbation. For more details on the strategy of the proof, we refer to \S\ref{section_strategy}. 
The outlook section \S \ref{section_discussion} is devoted to the robustness of our approach as well as its possible extension to periodic wave trains in viscous conservation laws where the handling of fully nonlocalized perturbations is similarly challenging as for the Lugiato-Lefever equation but for different reasons.
In case of the Lugiato-Lefever equation, a crucial difficulty towards extending to a fully nonlocalized stability result is to choose a suitable class of perturbations which contains all $C^\infty$-functions. The generic space of perturbations is given by $$C_{\textrm{ub}}^m(\R) = \{f: \R \rightarrow \C: f \textrm{ is }m\textrm{-times differentiable  with uniformly continuous derivatives}\},$$ $m\in \mathbb{N}$, whenever studying reaction-diffusion systems \cite{SHR, Bjoerns, ours}. The solutions of (\ref{LLE}) with initial data in $H^1_{\textrm{\textrm{per}}}(\R) \oplus H^1(\R)$ naturally lie in $C_{\textrm{ub}}(\R)$ due to Sobolev embedding. However, it is shown in \cite{bona} that $t \mapsto ||e^{i \partial_x^2 t}u_0||_{L^\infty}$ blows up in finite times for certain $u_0 \in C_{\textrm{ub}}(\R)$ and therefore it is convenient to study other spaces than $C_{\textrm{ub}}(\R)$ to approach a fully nonlocalized stability result for the Lugiato-Lefever equation  (\ref{LLE}). More suitable variants are given by the so-called modulation spaces $M^{m}_{\infty,1}(\R)$, $m\in \mathbb{N}_0$, which are introduced in \cite{mod_space}.  We discuss such an extension in \S\ref{section_modulation}.

\section{Preparation and main result}
We reformulate the Lugiato-Lefever equation as a semilinear system with a $\mathbb{C}$-linear part by splitting into real- and imaginary variables. Then, we construct perturbed solutions with initial data in $L^2_{\textrm{\textrm{per}}}(0,T) \oplus L^2(\R)$ and derive the associated perturbation equations. At the end of this section, we impose spectral properties and formulate our main result.
\subsection{Reformulation as real system}
As $|u|^2u$ is not differentiable with respect to $u \in \mathbb{C}$, we introduce  $\u := ( \u_r,\u_i)^T := (\Re (u), \Im (u))^T: \R \rightarrow \R^2$ which transforms (\ref{LLE}) into the real system
\begin{align}
\label{LLE_real}
    \u_t = \mathcal{J} \left(\begin{pmatrix}-\beta & 0 \\
    0 & -\beta \end{pmatrix}\u_{xx} + \begin{pmatrix}-\alpha & 0 \\
    0 & -\alpha \end{pmatrix}\u \right) - \u + \mathcal{N}(\u) + \begin{pmatrix}
        F \\
        0
    \end{pmatrix},
\end{align}
where
\begin{align*}
    \mathcal{J} = \begin{pmatrix}
        0 & -1 \\
        1 & 0
    \end{pmatrix}, \quad \mathcal{N}(\u) = |\u|^2 J \u = \begin{pmatrix}
        - \u_i^3 - \u_r^2 \u_i \\
        \u_r\u_i^2 + \u_r^3 
    \end{pmatrix}.
\end{align*}

\subsection{The perturbed solution in \texorpdfstring{$L^2_{\textrm{per}}(0,T) \oplus L^2(\R)$}{Lg}}
\label{section_construct_solution}
We assume the existence of a periodic standing wave.
\begin{itemize}
\item[\namedlabel{assH1}{\upshape (H1)}] There exists a smooth, nonconstant and $T$-periodic stationary solution $\phi_0: \R \rightarrow \mathbb{C}$  of (\ref{LLE}). 
\end{itemize}
We set $\phi:= (\phi_r,\phi_i)^T := (\Re (\phi_0), \Im (\phi_0))^T: \R \rightarrow \R^2$
and construct a solution of (\ref{LLE_real}) with initial datum $\phi + \w_0 + \v_0$ in $L^2_{\textrm{\textrm{per}}}(0,T) \oplus L^2(\R)$ with $\w_0 \in L^2_{\textrm{per}}(0,T)$ and $\v_0 \in L^2(\R)$ by following the strategy of \cite{kunst_klaus}.  
We first solve  
\begin{align}
\label{LLE_period}
    \begin{split} 
   & \w_t = \mathcal{J} \left(\begin{pmatrix}-\beta & 0 \\
    0 & -\beta \end{pmatrix}\w_{xx} + \begin{pmatrix}-\alpha & 0 \\
    0 & -\alpha \end{pmatrix}\w \right) - \w + \mathcal{N}(\w) + \begin{pmatrix}
        F \\
        0
    \end{pmatrix}\\
        &\w(0) =  \phi + \mathbf{w}_0
    \end{split}
\end{align}
in $L^2_{\textrm{\textrm{per}}}(0,T)$. Then, we transfer from the solution $\w$ of (\ref{LLE_real}) to a solution $\u$ of \ref{LLE_real} with initial datum $\phi + \w_0 + \v_0$ by solving the perturbed problem
\begin{align}
    \label{LLE_l2}
     \begin{split} 
   & \v_t = \mathcal{J} \left(\begin{pmatrix}-\beta & 0 \\
    0 & -\beta \end{pmatrix}\v_{xx} - \begin{pmatrix}-\alpha & 0 \\
    0 & -\alpha \end{pmatrix}\v \right) - \v + \mathcal{N}(\v + \w) - \mathcal{N}(\w)\\
        &\v(0) =  \mathbf{v}_0.
    \end{split}
\end{align}
In summary, if we solve (\ref{LLE_period}) and (\ref{LLE_l2}), then $\u = \w + \v$ is a solution of (\ref{LLE_real}) with $\u(0) = \phi +  \w_0 + \v_0$. 
 
\begin{remark}[Interpretation from fiber optics]
\label{remark_fiber}
    The cubic nonlinear Schr\"odinger equation on $L^2_{\textrm{\textrm{per}}}(0,T) \oplus L^2(\R)$ can be understood as model from nonlinear fiber optics by considering $t\geq 0$ as point on a fiber and $x \in \R$ as time variable, cf. \cite{application}. The stationary periodic solution $\phi$ is then the signal at any point on the fiber if it is chosen as input signal.  Prescribing a  signal at the initial point of the fiber by $\phi + L^2_{\textrm{\textrm{per}}}(0,T)$,  we ask how the signal as function depending on the time $x\in \R$ looks at the place $t\geq 0$ on the fiber. Adding an $L^2(\R)$-perturbation corresponds to temporally limited changes of the input signal. In particular, one may switch off the periodic signal for finitely many times as illustrated in Figure \ref{figure_missing_teeth}.
    Global existence of the perturbed solutions then translates to the observation that the fiber has infinite length while stability of $\phi$ is interpreted as that the signal stays close to $\phi$ at any point $t \geq 0$ on the fiber.
\end{remark}

\begin{figure}[h!]
\centering
\includegraphics[width=0.6\linewidth]{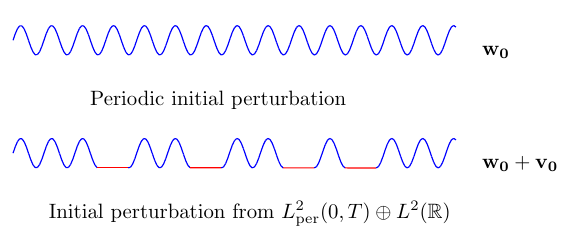}
\caption{For the sake of illustration, we reduce to the real part of an initial perturbation $\w_0 + \v_0$ with $\w_0 \in L_{\textrm{\textrm{per}}}^2(0,T)$ and $\v_0 \in L^2(\R)$. This figure demonstrates that $\v_0$ can in particular be chosen such that $\v_0 + \w_0$ coincides with $\w_0$ except for finitely many periods for which the signal vanishes, which explains the name tooth space for $L^2_{\textrm{\textrm{per}}}(0,T) \oplus L^2(\R)$ ("knocked out teeth").}
\label{figure_missing_teeth}
  \end{figure}
\subsection{Unmodulated perturbation equations}
Given a solution $\u(t) = \w(t) + \v(t)$ of (\ref{LLE_real}), we derive the unmodulated perturbation equations by
splitting the perturbation as
\begin{align*}
\tilde{\u}(t) = \u(t) - \phi = (\w(t) - \phi)+ \v(t)
\text{ and setting }  \tilde{\w}(t) = \w(t) - \phi.    
\end{align*}
This gives the coupled perturbation system
\begin{align}
\label{LLE_perturbed}
    \begin{split} 
   & \tilde{\w}_t = \mathcal{L}_0(\phi)\tilde{\w} + \mathcal{R}_1(\phi)(\tilde{\w})\\
        &\tilde{\w}(0) =  \mathbf{w}_0,
    \end{split}
    \begin{split} 
   & \v_t = \mathcal{L}_0(\phi)\v + \mathcal{R}_2(\phi)(\tilde{\w},\v)\\
        &\v(0) =  \v_0,
    \end{split}
\end{align}
where $\mathcal{L}_0(\phi)$ is the linearization of (\ref{LLE_real}) about $\phi$, given by
\begin{align}
\label{linearization}
     &\mathcal{L}_0(\phi) = \mathcal{J} \begin{pmatrix}-\beta \partial_x^2 - \alpha + 3\phi_r^2 + \phi_i^2 & 2\phi_r\phi_i \\
    2\phi_r\phi_i  & -\beta \partial_x^2 - \alpha + \phi_r^2 + 3\phi_i^2  \end{pmatrix}  - \mathcal{I},
    \end{align}
        first residual nonlinearity is given by
    \begin{align}
    \label{def_R_1}
    \mathcal{R}_1(\phi)(\tilde{\w}) &= \mathcal{N}(\tilde{\w} + \phi) - \mathcal{N}(\phi) - \mathcal{N}'(\phi)\tilde{\w},
    \end{align}
    and the second residual nonlinearity is defined by
    \begin{align}
    \label{def_R_2}
    &\mathcal{R}_2(\phi)(\tilde{\w}, \v) = \mathcal{R}_1(\phi)(\v+\tilde{\w}) - \mathcal{R}_1(\phi)(\tilde{\w})  =  \mathcal{R}_{2,1}(\phi)(\tilde{\w}, \v) + \mathcal{R}_{2,2}(\phi)(\tilde{\w}, \v),
    \end{align}
    with
    \begin{align*}
    \mathcal{R}_{2,1}(\phi)(\tilde{\w}, \v) = \mathcal{N}(\v+\tilde{\w} + \phi) &- \mathcal{N}(\tilde{\w}+\phi) - \mathcal{N}'(\tilde{\w} + \phi)\v, \quad  \mathcal{R}_{2,2}(\phi)(\tilde{\w}, \v) = \mathcal{N}'(\tilde{\w} + \phi)\v - \mathcal{N}'(\phi)\v.
\end{align*}
Fix some constant $K>0$.  Then, there exists a constant $C>0$ such that for $\v,\tilde{\w}\in \mathbb{C}$ with $|\v|, |\tilde{\w}|\leq K$ we have the nonlinear bounds
\begin{align}
\label{formal_nonlinear bounds}
    |\mathcal{R}_1(\phi)(\tilde{\w})| \leq C |\tilde{\w}|^2, \quad |\mathcal{R}_{2,1}(\phi)(\tilde{\w}, \v)| \leq C |\v|^2, \quad  |\mathcal{R}_{2,2}(\phi)(\tilde{\w}, \v)| \leq C |\v||\tilde{\w}|.
\end{align}
For the local wellposedness of (\ref{LLE_perturbed}), we refer to \S\ref{section_well_posed}.

We also abbreviate $\mathcal{L}_0 = \mathcal{L}_0(\phi)$, $\mathcal{R}_1 = \mathcal{R}_1(\phi)$, $\mathcal{R}_2 = \mathcal{R}_2(\phi)$, $\mathcal{R}_{2,1} = \mathcal{R}_{2,1}(\phi)$ and $\mathcal{R}_{2,2} = \mathcal{R}_{2,2}(\phi)$ whenever $\phi$ is the original periodic wave profile.
\subsection{Spectral assumptions on \texorpdfstring{$\phi$}{Lg}}
\label{section_spectrum}

Consider the Bloch operators
$\El(\xi) = e^{-i\xi \cdot} \mathcal{L}_0 e^{i\xi \cdot}$, $\xi\in[-\frac{\pi}{T},\frac{\pi}{T})$,   
posed on $L_{\mathrm{\textrm{per}}}^2(0,T)$ with domain $D(\El(\xi)) = H_{\mathrm{\textrm{per}}}^2(0,T)$. Since $\El(\xi)$ has compact resolvent, its spectrum consists of isolated eigenvalues of finite algebraic multiplicity only.

We introduce the standard \emph{diffusive spectral stability} assumptions, cf.~\cite{LLE_periodic, haragus, uniform,simon}.
\begin{itemize}
\setlength\itemsep{0em}
\item[\namedlabel{assD1}{\upshape (D1)}] We have $\sigma_{L^2}(\El_0)\subset\{\lambda\in\C:\Re(\lambda)<0\}\cup\{0\}$;
\item[\namedlabel{assD2}{\upshape (D2)}] There exists a constant $\theta>0$ such that for any $\xi\in [-\frac{\pi}{T},\frac{\pi}{T})$ we have $\Re\,\sigma_{L^2_{\textrm{per}}(0,T)}(\El(\xi))\leq-\theta \xi^2$;
\item[\namedlabel{assD3}{\upshape (D3)}] $0$ is a simple eigenvalue of $\El(0)$.
\end{itemize}
The spectrum of $\El_0$ on $L^2(\R)$ is the union of the spectra of the Bloch operators, i.e.,
\begin{align} \label{Blochspecdecomp}
\sigma_{L^2}(\El_0) = \bigcup_{\xi \in [-\frac{\pi}{T},\frac{\pi}{T})} \sigma_{L^2_{\textrm{per}}(0,T)}(\El(\xi)).
\end{align}
The spectrum $\sigma_{L^2}(\El_0)$ is purely essential, see e.g. \cite{Gardner, Kapitula}.

We emphasize that the periodic solutions of (\ref{LLE}) established in  \cite{HD} and \cite{BengeldeRijk} satisfy \ref{assH1} and \ref{assD1}-\ref{assD3}.

For the well-known consequences of the diffusive spectral stability assumptions, we refer to \cite[Lemma 2.1]{haragus} and references therein. First, Assumption~\ref{assD3} together with the translational invariance of (\ref{LLE}) imply that the kernel of $\El(0)$ is spanned by $\phi'$. Therefore, $0$ is also a simple eigenvalue of the adjoint operator $\El(0)^*$. By $\smash{\widetilde{\Phi}_0} \in H^2_{\mathrm{\textrm{per}}}(0,T)$, we denote the corresponding eigenfunction satisfying 
\begin{align*} 
\big\langle \widetilde{\Phi}_0,\phi'\big\rangle_{L^2(0,T)} = 1.\end{align*}

The spectral projection onto $\textrm{span}\{\phi'\}$ is given by
\begin{align*}
    \Pi(0)g = \phi'\langle\widetilde{\Phi}_0,g \rangle_{L^2(0,T)} , \quad g \in L^2_{\textrm{\textrm{per}}}(0,T)
\end{align*}
and we refer to \cite[Lemma 2.1]{haragus} for properties of $\sigma_{L^2}(\mathcal{L}_0)$ and  $\sigma_{L^2_{\textrm{per}}(0,T)}(\El(\xi))$.

\begin{remark}[Spectrum on modulation space]
The identity (\ref{Blochspecdecomp}) also holds on $C_{\textrm{ub}}(\R)$ and in contrast $\sigma_{C_{\textrm{ub}}}(\El_0)$ consists entirely of continuous point spectrum, \cite{Gardner, Kapitula}. Since 
\begin{align}
\label{embedding_mod_spaces}
    C_{\mathrm{ub}}^{m+2}(\R) \hookrightarrow M_{\infty,1}^m(\R)  \hookrightarrow C_{\mathrm{ub}}^{m}(\R),
\end{align} 
 for any $m\in \mathbb{N}_0$, one also finds (\ref{Blochspecdecomp}) on $M_{\infty,1}(\R)$ instead of $L^2(\R)$ or $C_{\mathrm{ub}}(\R)$ by \cite[2.17 Proposition]{nagel}. In particular, $\sigma_{M_{\infty,1}}(\El_0)$ also consists entirely of continuous point spectrum.

We briefly explain (\ref{embedding_mod_spaces}): from \cite{Klaus_thesis}, it is known that $C_{b}^2(\R) \hookrightarrow M_{\infty,1}(\R) \hookrightarrow C_{b}(\R)$ and it remains to show that $M_{\infty,1} \hookrightarrow C_{\textrm{ub}}(\R)$. This holds due to the observation that $C^\infty(\R)$ lies dense in $M_{\infty,1}(\R)$, $||f||_{L^\infty} \lesssim ||f||_{M_{\infty,1}}$ and since the uniform limit of uniformly continuous functions is uniformly continuous itself. Similarly, one argues for $m \geq 1$.
\end{remark}

\subsection{Formulation of main result}
\label{section_main_result}
We are now in the position to state our main result.
\begin{theorem}
\label{main_result}
Assume \ref{assH1} and \ref{assD1}-\ref{assD3}. There exist constants $C,\varepsilon>0$ such that for initial data $\w_0 \in H_{{\mathrm{per}}}^6(0,T)$ and $\v_0 \in H^3(\R)$ with 
$$E_0 := ||\w_0+\v_0||_{H^6_{\mathrm{per}}(0,T)\oplus H^3(\R)} < \varepsilon$$
there exist a unique solution 
\begin{align}
\label{properties_u_in_c_ub}
\u(t) 
 \in C([0,\infty); C^2_{\mathrm{ub}}(\R)) \cap C^1([0,\infty);C_{\mathrm{ub}}(\R))     
\end{align}
of (\ref{LLE_real}) with initial condition $\u(0) = \phi + \w_0 + \v_0$, 
 some smooth function $\gamma \in C([0,\infty), H^5(\R))$ and a constant $\sigma_* \in \R$ with the properties
\begin{align}
    \begin{split}
    \label{u_linfty_1}
           ||\u(t) - \phi||_{L^\infty} &\leq C E_0 
    \end{split}
    \end{align}
and
\begin{align}
\label{u_l_infty_2}
\begin{split}
    ||\u(\cdot,t) - \phi_0(\cdot + \sigma_* + \gamma(\cdot,t))||_{L^\infty} &\leq C (1+t)^{-\frac{3}{4}} E_0,
    \end{split}
\end{align}
for all $t\geq 0$.

\end{theorem}

 We briefly discuss the regularity assumptions in Theorem \ref{main_result}. The assumption $\v_0 \in H^3(\R)$ is justified by the fact that the regularity control on $\v$ mainly  proceeds along the lines of a standard $L^1\cap H^k$-nonlinear stability analysis \cite{haragus, JONZ} using the nonlinear damping estimate established in \cite{zumbrun_remark}. The reason for the regularity assumption on $\w_0$ is that $\v$  is considered as a perturbation of the periodic solution $\w$, which  yields expressions in the modulated perturbation equations for $\v$ where $\w$ appears with two spatial derivatives. As we need to control $\v$ in $H^3(\R)$, this leads to three more derivatives on $\w$. Therefore, we need to bound the fifth derivative in $L^\infty$ which is covered by assuming $\w_0 \in H_{\textrm{\textrm{per}}}^6(0,T)$.
Since we estimate the $H^1$-norm of the residual in Section \ref{section_l_infty_control} in order to find an $L^\infty$-estimate (in the spirit of the embedding $H^1(\R) \hookrightarrow L^\infty(\R)$), it suffices to demand $\w_0 \in H^5_{\textrm{per}}(0,T)$ and $\v_0 \in H^2(\R)$ when one only aims to establish pure $L^2$-estimates on $\v$   and henceforth obtain (\ref{u_l_infty_2}) in $L^\infty(\R)$ with lower decay rate $(1+t)^{-\frac{1}{2}}$. 

 Comparing (\ref{u_l_infty_2}) to the associated estimate in \cite{haragus}, we loose an algebraic decay factor of $\frac{1}{4}$ and one might ask whether we can compensate this lack of decay by taking the assumption $\v_0 \in L^1(\R)\cap H^3(\R)$. However, this assumption does not improve the decay rates due to the coupling terms $\mathcal{R}_2(\tilde{\w},\v)$ in (\ref{LLE_perturbed}) which cannot be controlled in $L^1(\R)$ but only in $L^2(\R)$, cf. Remark \ref{remark_on_decay_half}. 
 
 Another interesting contrast to the result in \cite{haragus} is the constant phase shift arising in the modulational estimate (\ref{u_l_infty_2}). The reason for this constant phase shift precisely originates from the fact the we not only enclose localized- but also co-periodic perturbation. This is reflected in designing the modulational approach to exploit the orbital stability result from \cite{LLE_periodic}. 

\begin{remark}[Uniqueness of solutions]
    We also comment on the uniqueness of solutions stated in Theorem \ref{main_result}. Concerning nonlocalized solutions of (\ref{LLE}), local well-posedness of (\ref{LLE}) with initial data from the modulation space $M_{\infty,1}^m(\R)$, $m\in \mathbb{N}_0$, can be established by precisely following the steps from \cite[Section 4.2]{Leonid_thesis} as the principal linear part of the equation (\ref{LLE}) generates a $C_0$-semigroup on $M_{\infty,1}(\R)$.
      That is, there exists a unique (mild) solution as element from $C([0,t), M_{\infty,1}(\R))$, $t>0$, whenever $\u_0\in C_{\textrm{ub}}^2(\R) \hookrightarrow M_{\infty,1}(\R)$. In particular, this shows that the solution $\u(t)$ of (\ref{LLE_real}) with $\u(0) = \w_0 + \v_0$ in Theorem \ref{main_result} satisfying (\ref{properties_u_in_c_ub}) is unique.
\end{remark}

\subsection{Strategy of proof}
\label{section_strategy}
The main task in the proof of Theorem \ref{main_result} is to find a suitable way to modulate the perturbations allowing to close a nonlinear argument through iterative estimates on their Duhamel formulae.  The construction of the perturbations $\hat{\w}(t)$ and $\v(t)$ shows that $\hat{\w}(t)$ is independent of $\v(t)$. Therefore, we first modulate $\hat{\w}(t)$ by introducing a temporal modulation function $\sigma(t)$. Precisely, we choose $\hat{\w}(x,t) = \w(x+\sigma(t),t) - \phi(x)$. Then, we do not only interpret $\v(t)$ as perturbation of $\w(t)$ but do the same for their modulated variants. This leads to the following inverse-modulated perturbation $$\hat{\v}(x,t) = \hat{\u}(x,t) - \hat{\w}(x,t) :=  \u(x-\sigma(t)-\gamma(t),t) - \hat{\w}(x,t) - \phi(x).$$ We arrive at a coupled system which allows to make an a-posteriori choice of the spatio-temporal phase modulation $\gamma: [0,\infty)\times\R \rightarrow \R^2$. We are then in the position to exploit the exponential decay of $|\sigma_t(t)|$ and $||\hat{\w}(t)||_{L^\infty}$.
We also introduce the forward-modulated perturbation 
\begin{align*}
\mathring{\v}(x,t) &= \mathring{\u}(x,t) - \hat{\w}(x+\sigma(t)+\gamma(x,t),t) \\
&:= \u(x,t)  - \phi(x+\sigma(t) + \gamma(x,t),t) - \hat{\w}(x+\sigma(t) + \gamma(x,t),t)
\end{align*}
which obeys a semilinear equation and can thus be used to control regularity in the nonlinear iteration. We control regularity by deriving a nonlinear damping estimate on $\mathring{\v}(t)$ and subsequently  relating $\mathring{\v}(t)$ to $\hat{\v}(t)$. 
Finally, we show the following stability estimates: for suitable constants $C,\delta>0$ and small initial conditions
\begin{align*}
    E_0 = E_p + E_l \textrm{ with } E_p = ||\w_0||_{H^6_{\textrm{per}}} \textrm{ and } E_l = ||\v_0||_{H^3},
\end{align*}
 we find
\begin{align}
\label{strategy_per_result}
    |\sigma_t(t)|, ||\hat{\w}(t)||_{H^6_{\textrm{per}}(0,T)}  &\leq C e^{-\delta t}E_p, \quad |\sigma| \leq C E_p, \quad t\geq 0,
\end{align}
and
\begin{align}
\label{l2_results}
   ||\gamma_x(t)||_{H^4}, ||\gamma_t(t)||_{H^3}, ||\mathring{\v}(t)||_{H^3}  \leq  C(1+t)^{-\frac{1}{2}}E_l, \quad 
    ||\gamma(t)||_{H^5} \leq C E_l, \quad t\geq 0. 
\end{align}
Using $L^\infty$-estimates on the propagators, we arrive at the refined $L^\infty$-estimate
    \begin{align}
    \label{l_infty_result}
    ||\mathring{\v}(t)||_{L^\infty}\leq C (1+t)^{-\frac{3}{4}} E_l, \quad t\geq 0.
\end{align}
   Applying  (\ref{strategy_per_result}) and (\ref{l_infty_result}), Theorem \ref{main_result} follows from  the observations
\begin{align*}
    ||\mathring{\u}(t)||_{L^\infty} &\leq ||\hat{\w}(\cdot + \sigma(t) + \gamma(\cdot,t),t)||_{L^\infty} + ||\mathring{\v}(t)||_{L^\infty}, \quad t\geq 0,
    \end{align*}
    and 
   \begin{align*}
    ||\u(t) - \phi||_{L^\infty}& \leq C\left(||\mathring{\u}(t)||_{L^\infty}  + |\sigma(t)| + ||\gamma(t)||_{L^\infty} \right), \quad t\geq 0,
\end{align*}
 and Sobolev embedding. The estimate (\ref{strategy_per_result}) can be derived along the lines of the co-periodic stability analysis in \cite{LLE_periodic} and therefore the remaining task of this paper is to prove (\ref{l2_results}) and (\ref{l_infty_result}).

\paragraph*{Notation.} Let $S$ be a set, and let $A, B \colon S \to \R$. Throughout the paper, the expression ``$A(x) \lesssim B(x)$ for $x \in S$'', means that there exists a constant $C>0$, independent of $x$, such that $A(x) \leq CB(x)$ holds for all $x \in S$. 

\paragraph*{Acknowledgments.} The author is grateful to Bj\"orn de Rijk for valuable discussions and support during all stages of this work.  This project is funded by the Deutsche Forschungsgemeinschaft (DFG, German Research Foundation) -- Project-ID 491897824.
\section{Linear estimates}
\label{section_linear}
We collect well-known facts about the semigroup generated by the linearization $\mathcal{L}_0$ on $L_{\textrm{per}}^2(0,T)$ and on $L^2(\R)$. 
\subsection{Semigroup decomposition  and estimates on \texorpdfstring{$L_{\textrm{per}}^2(0,T)$}{Lg}}
\label{section_linear_per}
Let $\chi \colon [0,\infty) \to \R$ be  a smooth temporal cut-off function  satisfying $\chi(t) = 0$ for $t \in [0,1]$ and $\chi(t) = 1$ for $t \in [2,\infty)$. 
We write \begin{align*}
   \tilde{S}_1(t) =  (e^{\mathcal{L}_0 t}- \chi(t)\Pi(0))\mathbf{g}
\end{align*}
and have the following linear estimate.
\begin{proposition}[\cite{LLE_periodic}, \cite{perkins21}]
\label{linear_exp}
    Assume \ref{assH1} and \ref{assD1}-\ref{assD3}. There exist constants $\delta_0,C>0$ such that the estimate
    \begin{align*}
        ||\tilde{S}_1(t) \mathbf{g}||_{H^6_{\textrm{per}}(0,T)} \leq C e^{-\delta_0 t}||\mathbf{g}||_{H^6_{\textrm{per}}(0,T)}
    \end{align*}
    is valid for all $\mathbf{g} \in {H^6_{\textrm{per}}(0,T)}$.
\end{proposition}
\subsection{Semigroup decomposition and estimates on \texorpdfstring{$L^2(\R)$}{Lg}}
\label{section_linear_l2}
Assume \ref{assH1} and \ref{assD1}-\ref{assD3}. 
Like in \cite{haragus}, we decompose
\begin{align*}
    e^{\El_0 t}\mathbf{g} = \tilde{S}_2(t)\mathbf{g} + \phi' s_p(t)\mathbf{g},
\end{align*}
with $s_p(t) = 0$ for $t\in [0,1]$. We have the following linear estimates.
\begin{proposition}
\label{prop_linear_l_2}
   Assume \ref{assH1} and \ref{assD1}-\ref{assD3}. Let $l,j \in \mathbb{N}_0$ and $k \in \{0,1,2\}$. Then there exists a constant $C_{l,j}>0$ such that
   \begin{align*}
       ||\partial_x^l \partial_t^j s_p(t) \partial_x^k\mathbf{g}||_{L^2} &\leq C_{l,j} (1+t)^{-\frac{l+j}{2}} ||\mathbf{g}||_{L^2}, \quad \mathbf{g} \in L^2(\R) \\
       ||\partial_x^l \partial_t^j s_p(t) \mathbf{g}||_{L^2} &\leq C_{l,j} (1+t)^{-\frac{1}{4}-\frac{l+j}{2}} ||\mathbf{g}||_{L^1}, \quad \mathbf{g} \in L^2(\R)\cap L^1(\R),
   \end{align*}
   for all $t\geq 0$. Furthermore, there exists a constant $C>0$ such that
   \begin{align*}
       ||\tilde{S}_2(t) \mathbf{g}||_{L^2} \leq  C (1+t)^{-\frac{3}{4}} ||\mathbf{g}||_{L^1\cap L^2}, \quad  \mathbf{g} \in L^2(\R)\cap L^1(\R)
   \end{align*}
   for all $t\geq  0$, and,
   \begin{align*}
       ||\tilde{S}_2(t) \mathbf{g}||_{L^2} \leq C(1+t)^{-\frac{1}{2}} ||\mathbf{g}||_{L^2}, \quad  \mathbf{g} \in L^2(\R),
   \end{align*}
   for all $t\geq 0$.
\end{proposition}
\begin{proof}
    The first two estimates are precisely shown in \cite[Lemma 3.1]{haragus}. The third one is a consequence of \cite[Lemma 3.1 \& Lemma 3.2]{haragus}. Adapting the proof of the estimate $\tilde{S}_c(t)$ in \cite[Lemma 3.2]{haragus}, one immediately finds the last estimate.
\end{proof}

\begin{proposition}[$L^\infty$-estimates]
\label{prop_linear_l_infty}
    Assume \ref{assH1} and \ref{assD1}-\ref{assD3}. There exist constants 
    $C,\delta_1>0$ such that
    \begin{align*}
        ||\tilde{S}_2(t)\mathbf{g}||_{L^\infty} &\leq C \left(e^{-\delta_1 t}||\mathbf{g}||_{H^1} + (1+t)^{-1}||\mathbf{g}||_{L^1 \cap L^2}\right), \quad \mathbf{g} \in H^1(\R) \cap L^1(\R), \\
        ||\tilde{S}_2(t)\mathbf{g}||_{L^\infty} &\leq C  (1+t)^{-\frac{3}{4}}||\mathbf{g}||_{H^1}, \quad \mathbf{g} \in H^1(\R),
                \end{align*}
                and
        \begin{align*}
        ||\partial_x s_p(t)\mathbf{g}||_{L^\infty} &\leq C (1+t)^{-\frac{3}{4}}||\mathbf{g}||_{L^2}, \quad \mathbf{g} \in L^2(\R), \\
        ||\partial_x s_p(t)\mathbf{g}||_{L^\infty} &\leq C (1+t)^{-1}||\mathbf{g}||_{L^1}, \quad \mathbf{g} \in L^1(\R),
    \end{align*}
    for all $t\geq 0$.
\end{proposition}
\begin{proof}
    These last two estimates are shown in  \cite[Lemma 3.2]{haragus}. The first two estimates are consequences of \cite[Lemma 3.1]{haragus} together with \cite[Corollary 3.4]{JONZ} and \cite[Proposition 3.1]{JNRZ}.
\end{proof}
\section{Nonlinear iteration scheme}
\subsection{Local existence of the solutions}
\label{section_well_posed}
Using standard semigroup theory, see e.g. \cite{cazenave} or \cite{pazy}, we establish

\begin{proposition}
\label{prop_existence_w}
    Let $\w_0 \in H^6_{\textrm{per}}(0,T)$. There exist a maximal time $T_{\textrm{max}} \in (0,\infty]$ and a unique solution $\w \in C([0,T_{\textrm{max}}); H^6_{\textrm{per}}(0,T))\cap  C^1([0,T_{\textrm{max}}); H^4_{\textrm{per}}(0,T))$ of (\ref{LLE_period}) with $\w(0) = \phi + \w_0$. If $T_{\textrm{max}}<\infty$, then 
    \begin{align}
    \label{blowup_crit_w}
        \limsup_{t\uparrow T_{\textrm{max}}}||\w(t)||_{H^4_{\textrm{per}}(0,T)} = \infty.
    \end{align}
\end{proposition}
Having the solution $\w$ at hand, the local existence of $\v$ follows.
\begin{proposition}
\label{prop_existence_unmodulated_v}
Let $\w$ and $T_{\textrm{max}}$ be as in Proposition \ref{prop_existence_w}.
    Let $\v_0 \in H^3(\R)$. There exist a maximal time $\tau_{\textrm{max}}\leq T_{\textrm{max}}$ and a unique solution $\v \in C([0,\tau_{\textrm{max}}); H^3(\R))\cap  C^1([0,\tau_{\textrm{max}}); H^1(\R))$ of (\ref{LLE_l2}) with $\v(0) = \v_0$. If $\tau_{\textrm{max}}<T_{\textrm{max}}$, then 
    \begin{align}
    \label{blowup_crit_v}
        \limsup_{t\uparrow \tau_{\textrm{max}}}||\v(t)||_{H^1} = \infty.
    \end{align}
\end{proposition}

\begin{proof}
Due to the embedding $H^1_{\textrm{per}}(0,T) \hookrightarrow L^\infty(\R)$, we have that 
\begin{align*}
    \v \mapsto \mathcal{N}( \v+ \w) - \mathcal{N}(\w)
\end{align*}
is locally Lipschitz-continuous as map from $H^1(\R)$ to $ H^1(\R)$. Moreover, 
\begin{align*}
     \mathcal{J} \left(\begin{pmatrix}-\beta & 0 \\
    0 & -\beta \end{pmatrix}\partial_x^2 - \begin{pmatrix}-\alpha & 0 \\
    0 & -\alpha \end{pmatrix} \right) - \mathcal{I} 
\end{align*}
generates a $C_0$-semigroup on $H^1(\R)$ with domain $H^3(\R)$. Therefore, Proposition \ref{prop_existence_unmodulated_v} follows from \cite[Theorem 1.4 of Section 6.1]{pazy}
\end{proof}
\subsection{Inverse-modulated perturbations}
\label{section_inverse_modulated}
We first modulate $\u(t)$, that is, we consider
\begin{align*}
    \u(x- \sigma(t),t) - \phi(x) = (\w(x- \sigma(t),t) - \phi(x)) + \v(x-\sigma(t),t) 
\end{align*}
for some $\sigma: [0,\infty) \rightarrow \R$ with $\sigma(0) = 0$ to be defined a-posteriori. Then, we set 
\begin{align}
\label{define_inverse_w}
    \hat{\w}(x,t) = \w(x- \sigma(t),t) - \phi(x).
\end{align}
Motivated by the fact that $\v(t)$ is a perturbation of $\w(t) = \tilde{\w}(t) + \phi$, we subsequently define
\begin{align}
\label{define_inverse_v}
    \hat{\v}(x,t) = \u(x- \sigma(t) - \gamma(x,t),t) - \hat{\w}(x,t) - \phi(x)
\end{align}
for some $\gamma:  \R \times [0,\infty) \rightarrow \R$ with $\gamma(\cdot,0) = 0$ to be defined a-posteriori.
We find the modulated perturbation equations for $\hat{\w}(t)$,
\begin{align}
\label{modulated_w}
\begin{split}
     &(\partial_t - \mathcal{L}_0)(\hat{\w} +  \phi'\sigma ) = \mathcal{R}_1(\hat{\w} ) - \sigma_t \hat{\w}_x \\
     &\hat{\w}(0) = \w_0 - \phi.
\end{split}
\end{align}
and the one for $\hat{\v}(t)$, 
\begin{align}
\label{modulated_v}
\begin{split}
     &(\partial_t-\mathcal{L}_0) (\hat{\v} + \phi'\gamma - \gamma_x \hat{\w} - \gamma_x \hat{\v}) = \mathcal{R}_3(\hat{\w},\hat{\v},\gamma) - \sigma_t \hat{\v}_x +  (1-\gamma_x) \mathcal{R}_{2,2}(\hat{\w},\hat{\v}) + \mathcal{T}(\hat{\w},\gamma)\\
     &\hat{\v}(0)=\v_0,
     \end{split}
\end{align}
where
\begin{align*}
   \mathcal{R}_3(\hat{\w},\hat{\v},\gamma) =  \mathcal{Q}(\hat{\w},\hat{\v},\gamma) + \partial_x\mathcal{S}(\hat{\v}, \gamma) + \partial_x^2\mathcal{P}(\hat{\v},\gamma)
\end{align*}
and
\begin{align*}
    &\mathcal{Q}(\hat{\w},\hat{\v},\gamma) = (1-\gamma_x) \mathcal{R}_{2,1}(\hat{\w},\hat{\v}), \\
     &\mathcal{S}(\hat{\v}, \gamma) = - \gamma_t\hat{\v} + \beta \mathcal{J}\left( \frac{\gamma_{xx}}{(1-\gamma_x)^2}\hat{\v} - \frac{\gamma_x^2}{1-\gamma_x}\phi' \right),
     \\
      &\mathcal{P}(\hat{\v}, \gamma) = -\beta \mathcal{J}\left(\gamma_x + \frac{\gamma_x}{1-\gamma_x}\right)\hat{\v}, \\
    &\mathcal{T}(\hat{\w},\gamma)= - \gamma_x \mathcal{R}_1(\hat{\w}) - \partial_x\left(\gamma_t \hat{\w} - \beta \mathcal{J}\left( \frac{\gamma_{xx}}{(1-\gamma_x)^2} \hat{\w}\right)\right) - \partial_x^2\left(\beta \mathcal{J}\left(\gamma_x + \frac{\gamma_x}{1-\gamma_x}\right) \hat{\w}\right).
\end{align*}
We delegate the derivation of (\ref{modulated_v}) to Appendix \ref{derivation_inv}.The main observation is that in the nonlinearities of (\ref{modulated_v})  any $\hat{\w}$- and $\sigma_t$-term is paired with a $\gamma_x$, $\hat{\v}$ or $\gamma_t$ contribution suggesting that we have sufficient control for an $L^2$-iteration scheme since we expect exponential decay for $||\hat{\w}(t)||_{L^\infty}$ and $|\sigma_t(t)|$ from \cite{LLE_periodic} while we control $\gamma_x(t)$, $\hat{\v}(t)$ and $\gamma_t(t)$ in $H^k(\R)$. To this end, we establish the following nonlinear bounds.

\begin{lemma}
\label{lemma_all_nonlinear_bounds}
    Fix a constant $c>0$ such that $||f||_{L^\infty} \leq \frac{1}{2}$ for all $f \in H^1(\R)$ with $||f||_{H^1} \leq c$. There exists a constant $C>0$ such that
\begin{align*}
   L^1\textrm{-bound: } ||\mathcal{R}_3(\hat{\w},\hat{\v},\gamma)||_{L^1} &\leq C \left(||\hat{\v}||_{L^2}^2 + ||(\gamma_x,\gamma_t)||_{H^2 \times H^1}(||\hat{\v}||_{H^2} + ||\gamma_x||_{L^2})\right), \\
    L^2\textrm{-bound: }||\mathcal{R}_3(\hat{\w},\hat{\v},\gamma)||_{L^2} &\leq C \left(||\hat{\v}||_{H^1}^2 + ||(\gamma_x,\gamma_t)||_{H^2 \times H^1}(||\hat{\v}||_{H^2} + ||\gamma_x||_{L^2})\right), \\
    H^1\textrm{-bound: }  ||\mathcal{R}_3(\hat{\w},\hat{\v},\gamma)||_{H^1} &\leq C \left(||\hat{\v}||_{H^1}^2 + ||(\gamma_x,\gamma_t)||_{H^3 \times H^2}(||\hat{\v}||_{H^3} + ||\gamma_x||_{H^1})\right),
    \end{align*}
and 
    \begin{align*}
   L^2\textrm{-bounds: } ||\sigma_t \hat{\v}_x||_{L^2} \leq C |\sigma_t||\hat{\v}||_{H^1}&, \quad 
    ||(1-\gamma_x) \mathcal{R}_{2,2}(\hat{\w}, \hat{\v})||_{L^2} \leq C ||\hat{\v}||_{L^2} ||\hat{\w}||_{H^1_{\textrm{per}}(0,T)}, \\
    ||\mathcal{T}(\hat{\w},\gamma)||_{L^2} &\leq C||(\gamma_x,\gamma_t)||_{H^2 \times H^1} ||\hat{\w}||_{H^3_{\textrm{per}}(0,T)},
     \\
   H^1\textrm{-bounds: } ||\sigma_t \hat{\v}_x||_{H^1} \leq C |\sigma_t||\hat{\v}||_{H^2}&, \quad
    ||(1-\gamma_x) \mathcal{R}_{2,2}(\hat{\w}, \hat{\v})||_{H^1} \leq C ||\hat{\v}||_{H^1} ||\hat{\w}||_{H^2_{\textrm{per}}(0,T)}, \\
    ||\mathcal{T}(\hat{\w},\gamma)||_{H^1} &\leq C||(\gamma_x,\gamma_t)||_{H^3 \times H^2} ||\hat{\w}||_{H^4_{\textrm{per}}(0,T)},
\end{align*}
hold for all $\hat{\v} \in H^3(\R)$, $\hat{\w} \in H^4_{\textrm{per}}(0,T)$, $(\gamma_t,\gamma_x) \in H^2(\R)\times H^3(\R)$ and $\sigma_t \in \R$ provided \begin{align*}
    ||\hat{\w}||_{H^4_{\textrm{per}}(0,T)}, ||\hat{\v}||_{H^1},||\gamma_x||_{H^3} \leq c.
\end{align*}
\end{lemma}

\subsection{Modulation in the purely co-periodic setting}

The forward-modulated perturbation $\mathring\w(x,t) = \w(x,t) - \phi(x+\sigma(t))$ fulfills the semilinear system 
\begin{align*}
    (\partial_t -\mathcal{L}_0) (\mathring{\w}(t) - \phi'\sigma) &= \mathcal{R}_4(\mathring{\w}(t), \sigma(t)) +  (\phi'(\cdot+ \sigma(t)) - \phi') \sigma_t(t)
\end{align*}
with  
\begin{align*}
    \mathcal{R}_4(\mathring{\w}(t), \sigma(t)) = \mathcal{R}_1(\phi(\cdot+\sigma(t)))(\mathring{\w}(t)) - (\mathcal{N}'(\phi) - \mathcal{N}'(\phi(\cdot+\sigma(t)))) \mathring{\w}(t).
\end{align*}
Introducing the temporal modulation function
\begin{align}
    \label{choice_of_sigma}
    \sigma(t) = \chi(t)\Pi(0) \tilde{\w}_0 + \int_0^t \chi(t-s) \Pi(0)\left(\mathcal{R}_4(\mathring{\w}(s), \sigma(s)) +(\phi'(\cdot+ \sigma(s)) - \phi') \sigma_s(s) \right)\,ds
\end{align}
gives rise to the Duhamel formula
\begin{align}
     \label{semilinear_w}
     \mathring{\w}(t) = \tilde{S}_1(t) \tilde{\w}_0 + \int_0^t \tilde{S}_1(t-s)\left(\mathcal{R}_4(\mathring{\w}(s), \sigma(s)) +(\phi'(\cdot+ \sigma(s)) - \phi') \sigma_s(s) \right)\,ds.
\end{align}
By a standard fixed point argument, we have local existence of $\sigma$.
\begin{proposition}
\label{local_existence_sigma}
Let $\w$ and $T_{\textrm{max}}$ be as in Proposition \ref{prop_existence_w}. There exists a maximal time $t_{\textrm{max}, \sigma} \leq T_{\textrm{max}}$ such that (\ref{choice_of_sigma}) with $\mathring\w(x,t) = \w(x,t) - \phi(x+\sigma(t))$ has a unique solution
    \begin{align*}
        \sigma\in C^1([0,t_{\textrm{max},\sigma});\R) \text{ with }  \sigma(0) = 0
\textrm{ and }
    |(\sigma(t),  \sigma_t(t))| <   \frac{1}{2},  \quad t\in [0,t_{\textrm{max},\sigma}).
\end{align*}
If $t_{\textrm{max},\sigma} < \tau_{\textrm{max}}$, then $
    \limsup_{t \uparrow t_{\textrm{max},\sigma}}|(\sigma(t),  \sigma_t(t))| \geq \frac{1}{2}$. 
\end{proposition}
Furthermore, the following nonlinear bound holds.
\begin{lemma}
\label{lemma_nonlinear_bounds_w}
   Let $K>0$. There exists  a constant $C>0$ such that
\begin{align*}
    ||\mathcal{R}_4(\mathring{\w}, \sigma)||_{H^6_{\textrm{per}}(0,T)} \leq C||\mathring{\w}||_{H^6_{\textrm{per}}(0,T)} \left(||\mathring{\w}||_{H^6_{\textrm{per}}(0,T)} + |\sigma| \right), \quad || (\phi'(\cdot+ \sigma - \phi') \sigma_t||_{H^6_{\textrm{per}}(0,T)} \leq C |\sigma||\sigma_t|,
\end{align*}
    provided 
    \begin{align*}
      |\sigma| + |\sigma_t| + ||\mathring{\w}||_{H^6_{\textrm{per}}(0,T)} \leq K.
    \end{align*}
\end{lemma}
\subsection{Choice of the spatial-temporal phase modulation}
\label{section_modualtion_functions_choices}
We have the Duhamel formula for $\hat{\v}(t)$,
\begin{align*}
    \hat{\v}(t) = e^{\mathcal{L}_0 t} \v_0 - \phi'\gamma(t) + \int_0^t e^{\mathcal{L}_0 (t-s)} \Bigl( &\mathcal{R}_3(\hat{\w}(s),\hat{\v}(s),\gamma(s)) \\
    &-\sigma_s(s)\hat{\v}_x(s) + (1-\gamma_x(s))\mathcal{R}_{2,2}(\hat{\v}(s),\hat{\w}(s)) + \mathcal{T}(\hat{\w}(s),\gamma(s)) \Bigr) \,ds\\
    +  \gamma_x(t) \hat{\w}(t) + \gamma_x(t) \hat{\v}(t),
\end{align*}
under the condition that $\gamma(0) = 0.$ 
We make the implicit choice
    \begin{align}
    \label{implicit_gamma}
    \begin{split}
    \gamma(t) = s_p(t)\v_0 + \int_0^t s_p(t-s)\Bigl(&\mathcal{R}_3(\hat{\w}(s),\hat{\v}(s),\gamma(s)) -\sigma_s(s)\hat{\v}_x(s) \\ &+ (1-\gamma_x(s))\mathcal{R}_{2,2}(\hat{\v}(s),\hat{\w}(s)) + \mathcal{T}(\hat{\w}(s),\gamma(s)) \Bigr)\,ds 
    \end{split}
\end{align}
which reduces the Duhamel formula for $\hat{\v}(t)$ to 
\begin{align}
\label{duhamel_inverse_v}
\begin{split}
    \hat{\v}(t) = \tilde{S}_2(t) \v_0 +  \gamma_x(t) \hat{\w}(t) + \gamma_x(t) \hat{\v}(t) + \int_0^t \tilde{S}_2(t-s) \Bigl( &\mathcal{R}_3(\hat{\w}(s),\hat{\v}(s),\gamma(s)) -\sigma_s(s)\hat{\v}_x(s) \\
    &+ (1-\gamma_x(s))\mathcal{R}_{2,2}(\hat{\v}(s),\hat{\w}(s)) + \mathcal{T}(\hat{\w}(s),\gamma(s)) \Bigr) \,ds.
\end{split}
\end{align}
Setting $t= 0$ in (\ref{implicit_gamma}) and using that $s_p(0) = 0$, one indeed verifies that $\gamma(0) = 0$.
\begin{proposition}
\label{prop_existence_modulation functions}
    Let $\w$ and $T_{\textrm{max}}$ as in Proposition \ref{prop_existence_w}, $\v$, $\v_0$ and $\tau_{\textrm{max}}$ as in Proposition  \ref{prop_existence_unmodulated_v} and $\sigma$ and $t_{\textrm{max},\sigma}$ as in Proposition \ref{local_existence_sigma}.
    Furthermore, let $0<c < \frac{1}{2}$ be a constant such that $||f||_{L^\infty} \leq \frac{1}{c}||f||_{H^1}$ for all $f \in H^1(\R)$.
 There exists a maximal time $t_{\textrm{max},\gamma} \leq \min\{\tau_{\textrm{max}},t_{\textrm{max},\sigma}\}$ such that (\ref{implicit_gamma}) with $\hat{\w}(t) = \w(\cdot -\sigma(t)) - \phi$ and $\hat{\v}(t) = \u(\cdot- \sigma(t) - \gamma(\cdot,t),t)- \hat{\w}(t)- \phi$ has a unique solution
    \begin{align*}
        \gamma \in C([0,t_{\textrm{max},\gamma}); H^5(\R)) \cap C^1([0,t_{\textrm{max},\gamma}); H^3(\R)) \text{ with }  \gamma(0) = 0
    \end{align*}
satisfying
\begin{align}
\label{max_gamma_cond}
    ||(\gamma(t),  \gamma_t(t))||_{H^5\times H^3} <  \frac{c}{2},  \quad t\in [0,t_{\textrm{max},\gamma}).
\end{align}
If $t_{\textrm{max},\gamma} < \min\{\tau_{\textrm{max}},t_{\textrm{max},\sigma}\}$, then
\begin{align}
\label{finite_time_cond_gamma}
    \limsup_{t \uparrow t_{\textrm{max},2}}||\gamma(t),  \gamma_t(t)||_{H^5 \times H^3} \geq \frac{c}{2}.
\end{align}
\end{proposition}
\begin{proof}
    See Appendix \ref{proof_local_mod}.
\end{proof}
\begin{corollary}
\label{cor_well_posed_inverse}
 Let $\w$ and $T_{\textrm{max}}$ as in Proposition \ref{prop_existence_w} and $\v$, $\v_0$ and $\tau_{\textrm{max}}$ as in Proposition \ref{prop_existence_unmodulated_v}. Let $\gamma$, $\sigma$ and $t_{\textrm{max},\sigma}$, $t_{\textrm{max},\gamma}$ as in Propositions \ref{local_existence_sigma} and \ref{prop_existence_modulation functions}. Then,  the inverse-modulated perturbation
  $\hat{\v} \in C([0,t_{\textrm{max},2}), L^2(\R))$ defined by (\ref{define_inverse_v}) satisfies (\ref{duhamel_inverse_v}) with $\hat{\w}(t) = \w(\cdot -\sigma(t)) - \phi$ and $\hat{\v}(t) \in H^3(\R)$  for all $t \in [0,t_{\textrm{max},\gamma})$.
\end{corollary}
\begin{proof}
   Let $t \in [0,t_{\textrm{max},\gamma})$. We observe
    \begin{align*}
     \hat{\v}(x,t)  =  \v(x- \sigma(t) - \gamma(x,t),t) + \w(x- \sigma(t) - \gamma(x,t),t) - \hat{\w}(x,t) - \phi(x) \\
     = \v(x- \sigma(t) - \gamma(x,t),t) + \w(x- \sigma(t) - \gamma(x,t),t) - \w(x - \sigma(t),t) 
    \end{align*}
    yielding on the one hand, with $ \w(t) \in H^5_{\textrm{per}}(0,T) \hookrightarrow W^{4,\infty}(\R)$ and the mean value theorem,
    \begin{align*}
        ||\w(x- \sigma(t) - \gamma(x,t),t) - \w(x - \sigma(t),t)||_{H^3} \lesssim ||\gamma(t)||_{H^3}.
    \end{align*}
    On the other hand, since $ \gamma(t) \in H^4 \hookrightarrow W^{3,\infty}$, $\sup_{s \in [0,t]}||\gamma_x(s)||_{L^\infty} \leq \frac{1}{2}$ and  $\v(t) \in H^3(\R)$, we conclude
    \begin{align*}
        \hat{\v}(t) \in H^3(\R)
    \end{align*}
    with the help of the chain and substitution rule.
\end{proof}

\begin{remark}
\label{remark_on_decay_half}
  We can now provide some intuition where the decay $(1+t)^{-\frac{1}{2}}$ for $\v$ and $\gamma_x$ in (\ref{l2_results}) originates from. For this purpose, assume that $|\sigma_t|$ admits exponential decay while $||\hat{\v}||_{L^2} \approx ||\hat{\v}_x||_{L^2}$ decays at rate $(1+t)^{-\kappa}$ with $\kappa\geq\frac{1}{2}$. Considering the term $||\sigma_t \hat{\v}_x||_{L^2} \leq ||\hat{\v}_x||_{L^2}|\sigma_t|$ in (\ref{implicit_gamma}) and (\ref{duhamel_inverse_v}), this yields the integral
    \begin{align*}
        \int_0^t (1+t-s)^{-\frac{1}{2}} (1+s)^{-\kappa}e^{-s} \,ds \lesssim (1+t)^{-\frac{1}{2}}.
    \end{align*}
    By noting that $\sigma_t \hat{\v}_x$ cannot be estimated in any $L^p$-norm with $1\leq p < 2$ due to the lack of localization of $\hat{\v}$, we can only close an iterative argument with $\kappa = \frac{1}{2}$ at best. This shows that even with the additional assumption $\v_0 \in L^1(\R)$, the decay rate on $\hat{\v}$ cannot be improved.
\end{remark}

\subsection{Forward-modulation of \texorpdfstring{$\v$}{Lg} and nonlinear damping estimates}
\label{section_foward_modulated}

We wish to control $||\hat{\v}||_{H^3}$ in terms of $||\hat{\v}||_{L^2}$, $\gamma_t$, $\gamma_x$, $\sigma_t$ and $\hat{\w}$ in the nonlinear iteration argument in order to control regularity, cf. Lemma \ref{lemma_all_nonlinear_bounds}. For this purpose, we introduce the forward-modulated perturbation $\mathring{\v}$, that is,
\begin{align}
\begin{split}
   \label{def_forward_v}
        \mathring{\v}(x,t) &= \u(x,t) - \hat{\w}(x+\sigma(t) + \gamma(x,t),t) - \phi(x+\sigma(t) + \gamma(x,t))  \\
    &=\u(x,t) - \w(x+ \gamma(x,t),t)\\
    &= \v(x,t) + \w(x,t) - \w(x+\gamma(x,t),t),
\end{split}
\end{align}
which satisfies the semilinear system, cf. Appendix \ref{derivation_forward},
\begin{align}
\label{forward_equation}
      (\partial_t - \mathcal{L}_0(\mathring{\phi})) \mathring{\v}(t) = \mathcal{R}_2(\mathring{\phi})(\tilde{\w}(\cdot+\gamma(\cdot,t),t),\mathring{\v}(t))+ \mathcal{R}_5(\tilde{\w}(t),\gamma(t),\gamma_t(t)) 
\end{align}
with 
\begin{align*}
    \mathcal{L}_0(\mathring{\phi}) = \mathcal{J} \begin{pmatrix}-\beta \partial_x^2 - \alpha + 3\mathring{\phi}_1^2 + \mathring{\phi}_2^2 & 2\mathring{\phi}_1\mathring{\phi}_2 \\
    2\mathring{\phi}_1\mathring{\phi}_2  & -\beta \partial_x^2 - \alpha + \mathring{\phi}_1^2 + 3\mathring{\phi}_2^2  \end{pmatrix}  - \mathcal{I}, \quad \mathring{\phi}(x,t) = \phi(x+ \gamma(x,t)),
\end{align*}
and
\begin{align*}
    &\mathcal{R}_5(\tilde{\w}(t),\gamma(t),\gamma_t(t)) \\
    &=-\tilde{\w}_x(\cdot+\gamma(\cdot,t),t) \gamma_t(t) -\phi'(\cdot+\gamma(\cdot,t)) \gamma_t(t) \\
    & \qquad - \beta\mathcal{J} \Bigl(\tilde{\w}_x(\cdot+\gamma(\cdot,t),t)\gamma_{xx}(t)  + \tilde{\w}_{xx}(\cdot+\gamma(\cdot,t),t)(2\gamma_x(t) + \gamma_x(t)^2) \\
    &\qquad \qquad \quad+ \phi'(\cdot+\gamma(\cdot,t))\gamma_{xx}(t)  + \phi''(\cdot+\gamma(\cdot,t))(2\gamma_x(t) + \gamma_x(t)^2) \Bigr).
\end{align*}
Note that $\mathcal{R}_2(\mathring{\phi})$ is as defined in (\ref{def_R_2}) and $\tilde{\w}(t) = \w(t)- \phi$ is the umodulated perturbation of $\w(t)$.
We setup the local existence for the forward-modulated perturbation $\mathring{\v}(t)$.
\begin{corollary}
\label{cor_existence_forward}
     Let $\w$ and $T_{\textrm{max}}$ as in Proposition \ref{prop_existence_w} and $\v$, and $\tau_{\textrm{max}}$ as in (\ref{prop_existence_unmodulated_v}). Let $\gamma$, $\sigma$ and $t_{\textrm{max},\sigma}$, $t_{\textrm{max},\gamma}$ as in Propositions \ref{local_existence_sigma} and \ref{prop_existence_modulation functions}. Then,  the forward-modulated perturbation
 $\mathring{\v} \in C([0,t_{\textrm{max},\gamma}), H^3(\R)) \cap C^1([0,t_{\textrm{max},\gamma}), H^1(\R))$ defined by (\ref{def_forward_v}) satisfies (\ref{forward_equation}) with $\tilde{\w}(t) = \w(t) - \phi$.
\end{corollary}
\begin{lemma}
\label{lemma_bounds_for_damping}
    Let $j = 1,2,3$. Fix $K>0$. There exists some $C>0$ such that for $t \in [0,t_{\textrm{max},\gamma})$, we obtain
    \begin{align*}
    ||\partial_x^j \mathcal{R}_5(\tilde{\w}(t),\gamma(t),\gamma_t(t))||_{L^2} &\leq C(||\gamma_x(t)||_{H^{j+1}} + ||\gamma_t(t)||_{H^{j}}),\\
      ||\partial_x^j\mathcal{R}_2(\mathring{\phi})(\tilde{\w}(\cdot + \gamma(\cdot,t),t), \mathring{\v}(t))||_{L^2} &\leq C||\mathring{\v}(t)||_{H^j}
    \end{align*}
    provided
    \begin{align*}
         \sup_{0\leq s \leq t} \left(||\tilde{\w}(s)||_{H^6_{\textrm{per}}(0,T)}+ ||\mathring{\v}(s)||_{H^3}+ ||(\gamma_x(s),\gamma_s(s))||_{H^4\times H^3}\right) \leq K.
    \end{align*}
\end{lemma}
\begin{proof}
We bound
\begin{align*}
  &||\partial_x^j( \beta \mathcal{J}\left(\tilde{\w}_{xx}(\cdot +  \gamma(\cdot,t),t)(2\gamma_x(t)+ \gamma_x(t)^2) + \tilde{\w}_x(\cdot + \gamma(\cdot,t),t)\gamma_{xx}(t)\right) + \tilde{\w}_x(\cdot + \gamma(\cdot,t),t)\gamma_t(t))||_{L^2} \\
  &\quad \lesssim ||\tilde{\w}(t)||_{H^{j+3}_{\textrm{per}}(0,T)} ||\gamma_x(t)||_{H^{j+1}} + ||\tilde{\w}(t)||_{H^{j+2}_{\textrm{per}}(0,T)} ||\gamma_t(t)||_{H^j},
\end{align*}
$j = 1,2,3$, where we used the Sobolev embedding $H^k_{\textrm{per}}(0,T) \hookrightarrow W^{k-1,\infty}(\R)$.
Similarly, one proceeds for terms involving $\phi$ instead of $\tilde{\w}$. Also it is straightforward to check the bound  for $\mathcal{R}_2$. 
\end{proof}
Proceeding as in \cite{haragus,uniform, zumbrun_remark}, we are now in the position to derive a nonlinear damping estimate for $\mathring{\v}$.
\begin{proposition}
\label{proposition_damping}
  Let $\gamma$, $\sigma$ and $t_{\textrm{max},\sigma}$, $t_{\textrm{max},\gamma}$ as in  Propositions \ref{local_existence_sigma} and  \ref{prop_existence_modulation functions}. Let $\w$ and $T_{\textrm{max}}$ as in Proposition \ref{prop_existence_w}. Take $\v$, $\v_0$ and $\tau_{\textrm{max}}$ as in Proposition \ref{prop_existence_unmodulated_v}. Define $\mathring{\v}$ through (\ref{def_forward_v}) and set $\tilde{\w}(t) = \w(t)- \phi$. Fix $K>0$. There exists a constant $C>0$ such that
  \begin{align}
  \begin{split}
  \label{eq_damping}
         ||\mathring{\v}(t)||^2_{H^3} \leq C \Bigl(e^{-t} ||\v_0||^2_{H^3}+ ||\mathring{\v}(t)||^2_{L^2} + \int_0^t e^{-(t-s)} \left(||\mathring{\v}(s)||_{L^2}^2 +  ||\gamma_x(s)||^2_{H^{4}} + ||\gamma_t(s)||^2_{H^{3}}\right) \,ds\Bigr)
  \end{split}
  \end{align}
  for all $t \in [0,t_{\textrm{max},\gamma})$ provided
  \begin{align}
  \label{preliminaries_boundind_damping}
      \sup_{0 \leq s\leq t} \left(  ||\tilde{\w}(s)||_{H^{6}_{\textrm{per}}(0,T)} + ||\mathring{\v}(s)||_{H^3} +  ||(\gamma_x(s),\gamma_s(s))||_{H^{3}\times H^2} \right) \leq K.
  \end{align}
\end{proposition}
\begin{proof}First we use that $\v_0 \in H^5(\R)$ gives a solution $\v \in C([0,\tau_{\textrm{max}})); H^5(\R))\cap  C^1([0,\tau_{\textrm{max}}); H^3(\R))$ 
 of (\ref{LLE_perturbed}) arguing analogously as in Proposition \ref{prop_existence_unmodulated_v}. Since $H^5(\R)$ is dense in $H^3(\R)$ an approximation argument as in \cite[Proposition 4.3.7]{cazenave} yields the result for $\v_0 \in H^3(\R)$.

Fix $K>0$. Let $t \in [0,t_{\textrm{max},
\gamma})$ such that (\ref{preliminaries_boundind_damping}) holds.   The forward-modulated perturbation $\mathring{\v}$ is designed such that the principal part in (\ref{forward_equation}) is given by $\partial_t - \mathcal{L}_0(\mathring{\phi})$. This is the reason why we choose the same energies as introduced in \cite{haragus}, \cite{uniform} and \cite{zumbrun_remark}, that is,
    \begin{align*}
        E_j(t) = ||\partial_x^j \mathring{\v}(t)||_{L^2}^2 - \frac{1}{2\beta} \langle \mathcal{J} M(\mathring{\phi})\partial_x^{j-1}\mathring{\v}, \partial_x^{j-1}\mathring{\v}\rangle, \quad j = 1,2,3,
    \end{align*}
    with 
    \begin{align*}
        M(\mathring{\phi}) = 2\begin{pmatrix}
            -2\mathring{\phi}_r \mathring{\phi}_i & \mathring{\phi}_r^2 - \mathring{\phi}_i^2 \\
            \mathring{\phi}_r^2 - \mathring{\phi}_i^2 & 2\mathring{\phi}_r \mathring{\phi}_i
        \end{pmatrix}.
    \end{align*}
    We compute, see \cite{zumbrun_remark} and \cite{uniform}, 
    \begin{align*}
        \partial_t E_j(t) = - 2E_j(t) + R_1(t) + R_2(t)
    \end{align*}
    with 
       \begin{align*}
        |R_1(t)| \leq \frac{2}{3} E_j(t) + C_1||\mathring{\v}(t)||_{L^2}^2
    \end{align*}
    for some $t$-independent constant $C_1>0$ and
    \begin{align*}
        R_2(t) = 2 \Re \, \langle \partial_x^j \mathcal{R}_5(\hat{\w}(t), \gamma(t),\gamma_t(t)), \partial_x^j \mathring{\v}(t) \rangle_{L^2} - \frac{1}{\beta} \Re \,\langle \mathcal{J}  M(\mathring{\phi})\partial_x^{j-1} \mathcal{R}_5(\hat{\w}(t), \gamma(t),\gamma_t(t)), \partial_x^{j-1} \mathring{\v}(t) \rangle_{L^2}.
    \end{align*}
    By Lemma \ref{lemma_bounds_for_damping}, using (\ref{preliminaries_boundind_damping}), interpolation and Young's inequality, there exists a $t$-independent constant $C_2>0$ such that
    \begin{align*}
        |R_2(t)| \leq \frac{1}{3} E_j(t) + C_2\left(||\mathring{\v}(t)||_{L^2}^2 +  ||\gamma_x(t)||^2_{H^{j+1}} + ||\gamma_t(t)||^2_{H^{j}}\right).
    \end{align*}
    We conclude \begin{align*}
        \partial_t E_j(t) \leq -E_j(t) + C_3\left(||\mathring{\v}(t)||_{L^2}^2 +  ||\gamma_x(t)||^2_{H^{j+1}} + ||\gamma_t(t)||^2_{H^{j}}\right)
    \end{align*}
    for some $t$-independent constant $C_3>0$ and $j = 1,2,3$. Integrating the latter and using, for some $t$-independent constant $C_4>0$,
    \begin{align*}
      ||\partial_x^j \mathring{\v}(t)||_{L^2}^2 \leq 2E_j(t) + C_4||\mathring{\v}(t)||_{L^2}^2,
    \end{align*}
    $j = 1,2,3$, which follow by interpolation and Young's inequality, we arrive at (\ref{eq_damping}).
    \end{proof}
\begin{lemma}
\label{lemma_relate}
  Fix $K>0$. Let the assumptions be as in the previous corollary. There exists $C>0$ such that for all $t \in[0,t_{\textrm{max},\gamma})$ it holds
    \begin{align}
    \label{relate_h2}
        ||\mathring{\v}(t)||_{L^2} \leq C \left( ||\hat{\v}(t)||_{L^2} + ||\gamma_x(t)||_{L^2} \right),
        \end{align}
        \begin{align}
        \label{relate_l2}
||\hat{\v}(t)||_{H^3} \leq C \left( ||\mathring{\v}(t)||_{H^3} +||\gamma_x(t)||_{H^{3}} \right)
    \end{align}
    and
    \begin{align}
    \label{relate_l_infty}
            || \mathring{\v}(t)||_{L^\infty} 
    \leq C\left( ||\hat{\v}(t)||_{L^\infty} + ||\gamma_x(t)||_{L^\infty} \right),
    \end{align}
    provided
    \begin{align*}
         \sup_{0\leq s \leq t} \left( ||\hat{\w}(s)||_{H^5_{\textrm{per}}} + ||\gamma(s)||_{H^4} + |\sigma
         (s)|\right)\leq K  \textrm{ and } \sup_{0\leq s\leq t}||\gamma_x(s)||_{L^\infty} \leq \frac{1}{2}.
    \end{align*}
\end{lemma}
\begin{proof}
  Let $t \in[0,t_{\textrm{max},\gamma})$.  
  We write $A_t(x) = x- \gamma(x,t) - \sigma(t)$ and  $B_t(x) = x+ \gamma(x,t) + \sigma(t)$ . By the inverse function theorem and the fact that $||\gamma
  _x(t)||_{L^\infty} \leq \frac{1}{2}$, it is easy to check that $A_t$ is invertible with
    \begin{align*}
      x =   A_t(A_t^{-1}(x)) = A_t^{-1}(x) - \gamma(A_t^{-1}(x),t) - \sigma(t)
    \end{align*}
    and therefore
    \begin{align*}
        A_t^{-1}(x) - B_t(x) &= \gamma(A_t^{-1}(x),t) - \gamma(x,t).
    \end{align*}
    We also compute
    \begin{align}
    \label{inverse_derivatives}
    \begin{split}
  & \partial_x(A_t^{-1}(x)) = \frac{1}{1-\gamma_x(A_t^{-1}(x),t)},  \quad \partial_x^2(A_t^{-1}(x)) = \frac{\gamma_{xx}(A_t^{-1}(x),t)}{(1-\gamma_x(A_t^{-1}(x),t))^3}, \\
   & \quad \partial_x^3(A_t^{-1}(x)) = \frac{\gamma_{xxx}(A_t^{-1}(x),t)}{(1-\gamma_x(A_t^{-1}(x),t))^4} + \frac{3\gamma_{xx}(A_t^{-1}(x),t)^2}{(1-\gamma_x(A_t^{-1}(x),t))^5}.
         \end{split}
    \end{align}
Using \begin{align}
\label{intermediate_integral}
   \gamma(A_t^{-1}(x),t) - \gamma(x,t) =  (A_t^{-1}(x)-x) \int_0^1 \gamma_x(x+ \theta (A_t^{-1}(x)-x), t) \,d\theta,
\end{align}
(\ref{inverse_derivatives}) and $||\gamma_x(t)||_{L^\infty} \leq \frac{1}{2}$,  we can estimate
\begin{align}
\label{relation_key_estimate}
     ||A_t^{-1}- B_t||_{L^2}  \lesssim ||\gamma_x||_{L^2}, \quad ||A_t^{-1}- B_t||_{H^l} \lesssim ||\gamma_x||_{H^l},
\end{align}
for $l = 1,2,3$. This is shown in \cite[ Lemma 2.7]{inv_paper} and \cite[Corollary 5.3]{zumbrun_remark}. 
It additionally follows
\begin{align*}
    ||\hat{\v}(A_t^{-1}(\cdot),t) - \hat{\v}(\cdot,t)||_{H^3} \lesssim ||\hat{\v}(t)||_{H^3}, \quad
    ||\mathring{\v}(A_t(\cdot),t) - \mathring{\v}(\cdot,t)||_{L^2} \lesssim ||\mathring{\v}(t)||_{L^2}
\end{align*}
as  for \cite[(3.24)\&(3.25)]{uniform}.
We observe
\begin{align*}
    \hat{\v}(A_t^{-1}(x),t) - \mathring{\v}(x,t) = (\hat{\w}(B_t(x),t) -\hat{\w}(A_t^{-1}(x),t)) + (\phi(A_t^{-1}(x)) - \phi(B_t(x)))
\end{align*}
and estimate with (\ref{relation_key_estimate}) and the mean value theorem,
\begin{align*}
    ||\hat{\v}(A_t^{-1}(\cdot),t) - \mathring{\v}(\cdot,t)||_{H^3} 
    \lesssim ||A_t^{-1} - B_t ||_{H^3} \lesssim ||\gamma_x||_{H^{3}}.
\end{align*}
On the other hand,
\begin{align*}
    \hat{\v}(x,t) - \mathring{\v}(A_t(x),t) = -\hat{\w}(x,t) + \hat{\w}(x+ \gamma(A_t(x),t)-\gamma(x,t),t) + (-\phi(x) + \phi(x+ \gamma(A_t(x),t)-\gamma(x,t)))
\end{align*}
and by and the mean value theorem and with 
\begin{align*}
    ||\gamma(A_t(\cdot),t) - \gamma(\cdot,t)||_{L^2} \lesssim ||\gamma_x(t)||_{L^2},
\end{align*}
which follows by (\ref{intermediate_integral}), 
we obtain
\begin{align*}
    ||\hat{\v}(\cdot,t) - \mathring{\v}(A_t(\cdot),t)||_{L^2} \lesssim ||\gamma_x(t)||_{L^2}.
\end{align*}
Putting everything together, we arrive at (\ref{relate_h2}) and (\ref{relate_l2}). The relation (\ref{relate_l_infty}) follows analogously.
\end{proof}

\section{Nonlinear stability analysis}
\label{section_nonlinear}
\subsection{The nonlinear iteration}
\label{section_iteration}
We first establish nonlinear stability of $\w$.
\begin{theorem}
\label{theorem_purely_periodic}
    Let $\w_0 \in H^6_{\textrm{per}}(0,T)$. Let $\w$, $T_{\textrm{max}}$ be as in  Proposition \ref{prop_existence_w} and let  $\sigma$ and $t_{\textrm{max},\sigma}$ as in Proposition \ref{local_existence_sigma}. There exist $C ,\delta_2, \varepsilon_p>0$ such that for \begin{align*}
        E_p := ||\w_0||_{H^6_{\textrm{per}}(0,T)}< \varepsilon_p , 
    \end{align*}
  the functions $\w(t)$ and $\sigma(t)$ exist globally, i.e. $t_{\textrm{max},\sigma} = T_{\textrm{max}} = \infty$, and 
    \begin{align*}
        |\sigma(t)|,||\w(t) - \phi||_{H^6_{\textrm{per}}(0,T)} \leq CE_p, \quad |\sigma_t(t)|, ||\hat{\w}(t)||_{H^6_{\textrm{per}}(0,T)}, ||\mathring{\w}(t)||_{H^6_{\textrm{per}}(0,T)} \leq C e^{-\delta_2 t} E_p,
    \end{align*}
    for all $t\geq 0$.
\end{theorem}
\begin{proof}
    The estimates on $|\sigma(t)|$, $|\sigma_t(t)|$ and $||\mathring{\w}(t)||_{H^6_{\textrm{per}}(0,T)}$ follow by a standard procedure, see  \cite{Kapitula} and \cite{LLE_periodic}, taking Proposition \ref{linear_exp} and Lemma \ref{lemma_nonlinear_bounds_w} for (\ref{semilinear_w}) into account. We observe that $$||\hat{\w}(t)||_{H^6_{\textrm{per}}(0,T)} = ||\hat{\w}(\cdot+\sigma(t),t)||_{H^6_{\textrm{per}}(0,T)} = ||\mathring{\w}(t)||_{H^6_{\textrm{per}}(0,T)} $$ for all $t\geq 0$, to finish the proof.
\end{proof}
Let $\w_0 \in H^6_{\textrm{per}}(0,T)$ and $\varepsilon_p$ as in Theorem \ref{theorem_purely_periodic}. We set $E_l = ||\v_0||_{H^3}$ and choose $0< c < \frac{1}{2}$ as in Proposition \ref{prop_existence_modulation functions}.
Under the assumption $E_p < \varepsilon_p$, we consider the template function $\eta: [0,t_{\textrm{max},\gamma}) \rightarrow \R$ given by
\begin{align*}
\begin{split}
        \eta(t) = \sup_{0\leq s\leq t} \left( (1+s)^{\frac{1}{2}}\left(||\mathring{\v}(s)||_{H^3(\R)} + ||(\gamma_x(s),\gamma_s(s))||_{H^4\times H^3}\right) + ||\gamma(t)||_{L^2}\right)
\end{split}
\end{align*}
and show that there exists a constant $C\geq 1$ independent of $E_l$ and $E_p$ such that
\begin{align}
\label{key_inequality}
    \eta(t) \leq C(E_l + \eta(t)^2 + \eta(t)E_p)
\end{align}
for every $t \in [0,t_{\textrm{max},\gamma})$ with $\eta(t) < \frac{c}{2}$. 
By Proposition \ref{prop_existence_modulation functions}, we have the property: if $t_{\textrm{max},\gamma} < \infty$, then 
\begin{align}
\label{final_blow_conditions}
    \limsup_{t \uparrow t_{\textrm{max},\gamma}} \eta(t) \geq \frac{c}{2}.
\end{align}
Furthermore, $\eta$ is continuous by Proposition \ref{prop_existence_modulation functions} and  Corollaries \ref{cor_well_posed_inverse} and  \ref{cor_existence_forward} and monotonically increasing.

\bigskip

\noindent\textbf{Iteration argument.} Suppose we have proven (\ref{key_inequality}). First we take $E_p < \min\{\varepsilon_p, \frac{1}{2C}\}$ which gives
\begin{align}
\label{absorbed_key_inequality}
    \eta(t) \leq 2C(E_l + \eta(t)^2).
\end{align}
Now, choose $4C^2E_l < \frac{c}{2}$. Assuming there exists some $t \in [0,t_{\textrm{max},\gamma})$ such that $\eta(t)\geq 4CE_l$, the continuity of $\eta$ provides a $t_0$ with $\eta(t_0)= 4CE_l < \frac{c}{2}$. Therefore, (\ref{absorbed_key_inequality}) and $c \in (0,\frac{1}{2})$ imply
\begin{align*}
    \eta(t_0) \leq 2C\left(E_l + (16C^2E_l) E_l\right) < 4CE_l.
\end{align*}
This is a contradiction and we arrive at
\begin{align}
\label{stability_result_end}
    \sup_{t \in [0,t_{\textrm{max},\gamma})} \eta(t) \leq 4CE_l < \frac{c}{2}
\end{align}
and hence (\ref{final_blow_conditions}) cannot hold. We conclude that (\ref{stability_result_end}) holds with  $t_{\textrm{max},\gamma} = \tau_{\textrm{max}} = \infty$. This proves (\ref{l2_results}) and it suffices to justify (\ref{key_inequality}). For this purpose, let $t \in [0,t_{\textrm{max},\gamma})$ with $\eta(t)< \frac{c}{2}$.

\bigskip
\noindent \textbf{Bound on $\hat{\v}$.}
We first bound $||\hat{\v}(s)||_{H^3}$ for which we use (\ref{relate_l2}), that is,
\begin{align}
\label{relate_v_hat}
    ||\hat{\v}(s)||_{H^3} \lesssim  (1+s)^{-\frac{1}{2}}\eta(s),
\end{align}
for $s\in [0,t]$.
Together with the nonlinear bounds, Lemma \ref{lemma_all_nonlinear_bounds}, Proposition \ref{prop_linear_l_2} and Theorem \ref{theorem_purely_periodic}, we arrive at
\begin{align}
\begin{split}
        ||\hat{\v}(s)||_{L^2} &\lesssim (1+s)^{-\frac{1}{2}} E_l +  ||\gamma_x(s)||_{L^2} ||\hat{\w}(s)||_{L^\infty} + ||\gamma_x(s)||_{H^1} ||\hat{\v}(s)||_{L^2} \\
    &\quad + \int_0^s (1+s-\tau)^{-\frac{3}{4}}||\mathcal{R}_3(\hat{\w}(\tau),\hat{\v}(\tau),\gamma(\tau))||_{L^1 \cap L^2} \,d\tau \\
     &\quad +\int_0^s (1+s-\tau)^{-\frac{1}{2}} \Bigl( ||\sigma_t(\tau)\hat{\v}_x(\tau)||_{L^2} + ||\mathcal{T}(\hat{\w}(\tau),\gamma(\tau))||_{L^2} \\
     &\qquad \qquad \qquad \qquad \quad\quad \quad + ||(1-\gamma_x(\tau))\mathcal{R}_{2,2}(\hat{\w}(\tau), \hat{\v}(\tau))||_{L^2}  \Bigr) \,d \tau \\
     &\lesssim (1+s)^{-\frac{1}{2}} E_l +  (1+s)^{-\frac{1}{2}}\eta(t)^2 + (1+s)^{-\frac{1}{2}}\eta(t)E_p \\
     & \quad +\eta(t)^2\int_0^s (1+s-\tau)^{-\frac{3}{4}} (1+\tau)^{-1} d\,\tau \\
     &\quad +  \eta(t)E_p \int_0^s (1+s-\tau)^{-\frac{1}{2}} (1+\tau)^{-\frac{1}{2}} e^{-\delta_2 \tau} d\, \tau \\
     &\lesssim (1+s)^{-\frac{1}{2}}\left( E_l+ \eta(t)^2 + \eta(t)E_p \right),
     \label{nonlinear_estimate_hat_v}
\end{split}
\end{align}
for $s\in [0,t]$, where we used that $\eta(t) \leq \frac{1}{2}$.
\bigskip

\noindent \textbf{Bounds on $\gamma$.}
We estimate, with (\ref{relate_v_hat}), Proposition \ref{prop_linear_l_2}, Lemma \ref{lemma_all_nonlinear_bounds} and Theorem \ref{theorem_purely_periodic},
\begin{align}
\begin{split}
        ||\partial_s^k \partial_x^l \gamma(s)||_{L^2} &\leq ||\partial_s^k \partial_x^l s_p(s)||_{L^2 \rightarrow L^2}||\v_0||_{L^2} + \int_0^s ||\partial_s^k \partial_x^l s_p(s-\tau)||_{L^1\rightarrow L^2}||\mathcal{R}_3(\hat{\w}(\tau),\hat{\v}(\tau),\gamma(\tau))||_{L^1} \,d\tau \\
     &\quad +\int_0^s ||\partial_s^k \partial_x^l s_p(s-\tau)||_{L^2\rightarrow L^2} \Bigl( ||\sigma_t(\tau)\hat{\v}_x(\tau)||_{L^2} + ||\mathcal{T}(\hat{\w}(\tau),\gamma(\tau))||_{L^2}\\
     &\qquad \qquad \qquad \qquad \qquad \qquad \qquad \quad + ||(1-\gamma_x(\tau))\mathcal{R}_{2,2}(\hat{\w}(\tau), \hat{\v}(\tau))||_{L^2} \Bigr) \,d \tau \\
     &\lesssim (1+s)^{-\frac{k+l}{2}} E_l + \eta(s)^2\int_0^s (1+s-\tau)^{-\frac{1}{4} - \frac{k+l}{2}} (1+\tau)^{-1} d\,\tau \\
     & \qquad\qquad \qquad\qquad+ \eta(s)E_p\int_0^s (1+s-\tau)^{-\frac{k+l}{2}} (1+\tau)^{-\frac{1}{2}} e^{-\delta_2 \tau} d\,\tau \\
     &\lesssim (E_l +  \eta(t)^2 + \eta(t)E_p) \begin{cases}
     1, & \textrm{ if } l + k = 0\\
         (1+s)^{-\frac{1}{2}}, & \textrm{ if } l + k = 1\\
         (1+s)^{-1}, & \textrm{ otherwise,}
     \end{cases}
     \label{nonlinear_estimate_gamma}
\end{split}
\end{align}
for every $s \in [0,t]$.

\bigskip

\noindent\textbf{Bounds on $\mathring{\v}$.}
Invoking (\ref{nonlinear_estimate_hat_v}) and (\ref{nonlinear_estimate_gamma}), (\ref{relate_h2}) yields
\begin{align*}
    ||\mathring{\v}(s)||_{L^2} \lesssim  (1+s)^{-\frac{1}{2}}\left( E_l+ \eta(t)^2 + \eta(t)E_p\right)
\end{align*}
for all $0\leq s \leq t$.
Finally, with Proposition \ref{proposition_damping} and (\ref{nonlinear_estimate_gamma}), we obtain
\begin{align*}
    ||\mathring{\v}(s)||_{H^3} \lesssim(1+s)^{-\frac{1}{2}}\left( E_l+ \eta(t)^2 + \eta(t)E_p\right),
\end{align*}
for all $0\leq s \leq t$, using $\eta(t) \leq \frac{1}{2}$.

We have shown the key inequality (\ref{key_inequality}).

\begin{remark}
\label{milder_decay}
    The proof of (\ref{key_inequality}) reveals that the nonlinear iteration argument closes  as long as $|\sigma_t(t)|$ and  $||\hat{\w}(t)||_{H^6_{\textrm{per}}(0,T)} $ decay of order $(1+t)^{-\kappa}$ for some $\kappa >\frac{1}{2}$.
\end{remark}

\subsection{Refined \texorpdfstring{$L^\infty$}{Lg}-estimates}
\label{section_l_infty_control}

By Proposition \ref{prop_linear_l_infty} 
and the nonlinear bounds in Lemma \ref{lemma_all_nonlinear_bounds} and inserting the estimates (\ref{l2_results}), we obtain
\begin{align*}
    ||\gamma_x(t)||_{L^\infty} &\leq ||\partial_x s_p(t)||_{L^2 \rightarrow L^\infty}||\v_0||_{L^2} + \int_0^t ||\partial_x s_p(t-s)||_{L^1\rightarrow L^\infty}||\mathcal{R}_3(\hat{\w}(s),\hat{\v}(s),\gamma(s))||_{L^1} \,ds \\
     &+\int_0^t ||\partial_x s_p(t-s)||_{L^2\rightarrow L^\infty} \Bigl( ||\sigma_t(s)\hat{\v}_x(s)||_{L^2} + ||\mathcal{T}(\hat{\w}(s),\gamma(s))||_{L^2} \\
     &\qquad \qquad \qquad \qquad \qquad \qquad  + ||(1-\gamma_x(s))\mathcal{R}_{2,2}(\hat{\w}(s), \hat{\v}(s))||_{L^2}  \Bigr) \,ds \\
     &\lesssim (1+t)^{-\frac{3}{4}} E_l  + E_l \left(\int_0^t (1+t-s)^{-1} (1+s)^{-1} \,ds + \int_0^t (1+t-s)^{-\frac{3}{4}} (1+s)^{-\frac{1}{2}} e^{-\delta_2 s} \,ds \right) \\
     &\lesssim (1+t)^{-\frac{3}{4}} E_l,
\end{align*}
for all $t\geq 0$. 
Furthermore, we have
\begin{align*}
    ||\hat{\v}(t)||_{L^\infty} &\lesssim (1+t)^{-\frac{3}{4}} E_l +  \int_0^t (1+t-s)^{-1}||\mathcal{R}_3(\hat{\w}(s),\hat{\v}(s),\gamma(s))||_{L^1 \cap L^2} \,ds \\
    &\quad + \int^t_0 e^{-\delta_1(t-s)} ||\mathcal{R}_3(\hat{\w}(s), \hat{\v}(s),\gamma(s))||_{H^1} \,ds
    \\
    & \quad +\int_0^t (1+t-s)^{-\frac{3}{4}} \Bigl( ||\sigma_t(s)\hat{\v}_x(s)||_{H^1} \\
     &\qquad \qquad \qquad \qquad \qquad   + ||(1-\gamma_x(s))\mathcal{R}_{2,2}(\hat{\w}(s), \hat{\v}(s))||_{H^1} + ||\mathcal{T}(\hat{\w}(s),\gamma(s))||_{H^1} \Bigr) \,ds \\
     &\lesssim (1+t)^{-\frac{3}{4}} E_l  +E_l \left(\int_0^t (1+t-s)^{-1}(1+s)^{-1} \,ds +\int_0^t (1+t-s)^{-\frac{3}{4}}(1+s)^{-\frac{1}{2}}e^{-\delta_2 s} \,ds \right) \\
     &\lesssim (1+t)^{-\frac{3}{4}} E_l,
\end{align*}
for all $t\geq 0$.
With the help of (\ref{relate_l_infty}),
we deduce
\begin{align*}
    ||\mathring{\v}(t)||_{L^\infty} \lesssim (1+t)^{-\frac{3}{4}} E_l,
\end{align*}
for all $t\geq 0$.
Finally, we observe that for $\sigma_* := \int_0^\infty \sigma_s(s) \,ds$, we have
\begin{align*}
    |\sigma_*-\sigma(t)| \leq \int_t^\infty |\sigma_s(s) | \,ds \lesssim  e^{-\delta_1 t}E_p
\end{align*}
and therefore
\begin{align*}
    ||\phi(\cdot + \sigma_* + \gamma(\cdot,t), t) - \phi(\cdot + \sigma(t) + \gamma(\cdot,t),t)||_{L^\infty} \lesssim e^{-\delta_1 t}E_p,
\end{align*}
for all $t\geq 0$. As described in Section \ref{section_strategy}, Theorem \ref{main_result} follows.

\section{Discussion}
\label{section_discussion}
\subsection{Applicability of the scheme to viscous conservation laws}
\label{section_conservation}
A satisfying $L^\infty$-stability theory considering $C_{\textrm{ub}}$-perturbations is established in parabolic reaction-diffusion systems \cite{Bjoerns} which can be extended beyond the parabolic setting as shown in the context of the FitzHugh-Nagumo system \cite{ours}.  On the other hand, there are crucial obstacles in establishing pure $L^\infty$-stability results for viscous conservation laws as described in \cite{SHR, ours}. An interesting ingredient of the $L^2$-analysis in \cite{inv_paper} is that the authors introduce a sum of spatio-temporal modulation functions to capture more than one critical mode arising from the presence of the conservation laws. Therefore, a key difficulty lies in the fact that the critical dynamics is governed by a coupled Whitham system for which one cannot immediately apply the Cole-Hopf transform as done in the scalar case for which this Whitham system reduces to the Burgers' equation \cite{Bjoerns, ours}.  This begs the question of whether we can apply our $L^2_{\textrm{per}}(0,T)\oplus L^2(\R)$-scheme to this setting, cf. \cite{inv_paper} and \cite{STVenant1}.  In order to generate a result like Theorem \ref{main_result} in this setting, the  question reduces to whether such a sum of spatio-temporal modulation functions is compatibel with our approach. As the $L^2_{\textrm{per}}(0,T)$-theory can be settled by a standard procedure, we strongly expect that it is relatively straightforward to allow for $L^2_{\textrm{per}}(0,T)\oplus L^2(\R)$-perturbations by precisely following our nonlinear analysis, in particular using the modified versions of inverse- and forward-modulated perturbations in \S \ref{section_inverse_modulated} and \S \ref{section_foward_modulated}, and by respecting the semigroup decomposition and estimates established in \cite[pp. 149]{inv_paper} and the nonlinear damping estimate in \cite[pp. 146]{inv_paper}.

\subsection{Uniformly subharmonic plus localized perturbations} 
\label{section_subharmonic}
Considering the uniformly subharmonic nonlinear stability result for the Lugiato-Lefever equation in \cite{uniform}, the question arises whether we can show a nonlinear stability result involving perturbations $\w_0 + \v_0$ with $\w_0 \in H_{\textrm{per}}^6(0,NT) \cap L^1_{\textrm{per}}(0,NT)$, $\v_0 \in H^3(\R)$ and which is uniform in $N \in \mathbb{N}$. We note that as in \cite{uniform} the additional condition $L^1_{\textrm{per}}(0,NT)$ is crucial in \cite{uniform} to guarantee the uniformity in $N$. 

Suppose \ref{assH1} and \ref{assD1}-\ref{assD3}.
Now, given a solution $\u = \w+ \v$ of \ref{LLE} where $\w$ solves (\ref{LLE_period}) and $\v$ solves (\ref{LLE_l2}) with $\u(0) = \w_0 + \v_0$, we broadly sketch a possible scheme for the claimed stability estimate
\begin{align}
    \label{uniform_result}
    ||\u(t)||_{L^\infty} \lesssim (1+t)^{-\frac{3}{4}} \left(||\w_0||_{H^6_{\textrm{per}}(0,NT) \cap L^1_{\textrm{per}}(0,NT)} +||\v_0||_{ H^3(\R)}\right), \quad t\geq 0,
\end{align}
uniformly in $N \in \mathbb{N}$.
Keeping the main lines of this paper, we might introduce the inverse-modulated perturbations
\begin{align*}
    \hat{\w}(x,t) &= \w(x- \sigma(t) - \gamma_1(x,t),t) - \phi, \\
    \hat{\v}(x,t) &= \u(x- \sigma(t) -\gamma_1(x,t) - \gamma_2(x,t),t) - \hat{\w}(x,t) - \phi(x)
\end{align*}
and the forward-modulated perturbations
\begin{align*}
    \mathring{\w}(x,t) &= \w(x,t) - \phi(x+ \sigma(t) + \gamma_1(x,t)), \\ 
    \mathring{\v}(x,t) &= \u(x,t) - \hat{\w}(x+ \sigma(t) +\gamma_1(x,t) + \gamma_2(x,t),t) - \phi(x+ \sigma(t) +\gamma_1(x,t) + \gamma_2(x,t)),
\end{align*}
 with modulation functions $$\sigma: [0,\infty) \rightarrow \R, \gamma_1 : [0,\infty) \rightarrow L^2_{\textrm{per}}(0,NT) \textrm{ and } \gamma_2: [0,\infty) \rightarrow L^2(\R).$$  
To derive suitable perturbation equations for $\hat{\v}$ and $\mathring{\v}$, we refer to Appendices \ref{derivation_inv} and \ref{derivation_forward}.
Finally, with the established decay rates \cite[Theorem 1.4]{uniform}, that is, $$||\partial_x\gamma_1(t)||_{L^2_{\textrm{per}}(0,NT)} , |\sigma_t(t)|, ||\hat{\w}(t)||_{L^2_{\textrm{per}}(0,NT)} \sim O((1+t)^{-\frac{3}{4}})$$ uniformly in $N$, one may invoke Remark \ref{milder_decay} and  \cite[Proposition 3.7]{uniform}). It remains to prove versions of Proposition \ref{proposition_damping} and  Lemma \ref{lemma_relate} to deduce (\ref{uniform_result}).   

\subsection{Nonlocalized phase modulations}
An interesting feature of the modulational stability estimate (\ref{u_l_infty_2}) is that we capture the most critical dynamics of the periodic wave by a phase modulation from $\R + H^3(\R)$. This class of functions covers the simplest nontrivial nonlocalized phase perturbations one could think of. Therefore,  an interesting open question is whether we can allow for perturbations from $L^2_{\textrm{per}}(0,T) \oplus L^2(\R)$ as well as ($L^\infty$-large) initial modulations $\gamma_0 -\sigma_*$ such that $||\gamma_0'||_{ H^3}$ is small to conclude an estimate such as (\ref{u_l_infty_2}). We refer to \cite{zumbrun_remark} for a nonlinear stability result against localized perturbations in this context, allowing for a nonlocalized initial phase modulation.

\subsection{Fully nonlocalized perturbations}
\label{section_modulation}
Local well-posedness results in nonlinear Schr\"odinger-type equations, possibly with (periodic) potentials, have been established for initial data  from the modulation space $M^{m}_{\infty,1}(\R)$, see e.g. \cite{kato,Leonid_thesis}, \cite[pages 245-252]{book_part_by_kato} and references therein.  At the same time, global in time results are widely open. Due to the additional dissipativity  in (\ref{LLE}) compared to the classical cubic nonlinear Schr\"odinger equation, we expect the possibility to extend our result to perturbations from $M^m_{\infty,1}(\R)$ for sufficiently large $m \in \mathbb{N}$.
This might be achieved by conceptually  following the lines of \cite{ours}. Nevertheless, it turns out that high-frequency damping and regularity control are delicate challenges and it seems that essentially new ideas have to be developed.

\appendix

\section{Derivation of (\ref{modulated_v})}
\label{derivation_inv}
We write  $\hat{\u}(x,t) = \u(x- \sigma(t) - \gamma(x,t),t) - \phi(x)$, insert (\ref{modulated_w}) and \cite[(3.4)]{uniform} to obtain
\begin{align*}
    (\partial_t-\mathcal{L}_0) (\hat{\v}) &= (\partial_t-\mathcal{L}_0)\hat{\u} - (\partial_t-\mathcal{L}_0) \hat{\w} \\
    &= - (\partial_t-\mathcal{L}_0)(\phi'\gamma) + (1-\gamma_x)\mathcal{R}_1(\hat{\u}) -\mathcal{R}_1(\hat{\w} ) + \sigma_t \hat{\w}_x \\
    &\qquad + \partial_x\mathcal{S}(\hat{\u}, \gamma,\gamma_t,\sigma_t) + \partial_x^2\mathcal{P}(\hat{\u},\gamma) + (\partial_t -\mathcal{L}_0)(\gamma_x \hat{\u}),
\end{align*}
where, recalling $ \hat{\u} = \hat{\v} + \hat{\w}$,
\begin{align*}
    \widetilde{\mathcal{S}}(\hat{\u}, \gamma,\gamma_t,\sigma_t) &= - \gamma_t\hat{\u} - \sigma_t \hat{\u} + \beta \mathcal{J}\left( \frac{\gamma_{xx}}{(1-\gamma_x)^2}\hat{\u} - \frac{\gamma_x^2}{1-\gamma_x}\phi' \right) \\
    &=  \widetilde{\mathcal{S}}(\hat{\v}, \gamma,\gamma_t,\sigma_t) - \gamma_t \hat{\w} - \sigma_t \hat{\w} + \beta \mathcal{J}\left( \frac{\gamma_{xx}}{(1-\gamma_x)^2} \hat{\w}\right)\\
    \mathcal{P}(\hat{\u}, \gamma) &= -\beta \mathcal{J}\left(\gamma_x + \frac{\gamma_x}{1-\gamma_x}\right) \hat{\u} = \mathcal{P}(\hat{\v}, \gamma)  + \mathcal{P}(\hat{\w}, \gamma) .
\end{align*}
We emphasize that the critical terms $\sigma_t \hat{\w}$ and $(\partial_t - \mathcal{L}_0)(\phi' \sigma)$ cancel out.
We arrive at
\begin{align*}
     (\partial_t-\mathcal{L}_0) (\hat{\v} + \phi'\gamma - \gamma_x\hat{\w} - \gamma_x\hat{\v}) = & \mathcal{Q}(\hat{\w},\hat{\v},\gamma) + \partial_x\mathcal{S}(\hat{\v}, \gamma)  + \partial_x^2\mathcal{P}(\hat{\v},\gamma)\\
     & - \sigma_t \hat{\v}_x + (1-\gamma_x)\mathcal{R}_{2,2}(\hat{\w},\hat{\v}) + \mathcal{T}(\hat{\w},\gamma) 
\end{align*}
with 
\begin{align*}
    \mathcal{Q}(\hat{\w},\hat{\v},\gamma) &= (1-\gamma_x) (\mathcal{R}_1(\hat{\v} + \hat{\w}) - \mathcal{R}_1(\hat{\w}) - \mathcal{R}_{2,2}(\hat{\w},\hat{\v})) = (1-\gamma_x) \mathcal{R}_{2,1}(\hat{\v}, \hat{\w}), \\
     \mathcal{S}(\hat{\v}, \gamma) &= - \gamma_t\hat{\v} + \beta \mathcal{J}\left( \frac{\gamma_{xx}}{(1-\gamma_x)^2}\hat{\v} - \frac{\gamma_x^2}{1-\gamma_x}\phi' \right), \\
    \mathcal{T}(\hat{\w},\gamma)&= - \gamma_x \mathcal{R}_1(\hat{\w}) - \partial_x\left(\gamma_t \hat{\w} - \beta \mathcal{J}\left( \frac{\gamma_{xx}}{(1-\gamma_x)^2} \hat{\w}\right)\right) - \partial_x^2\left(\beta \mathcal{J}\left(\gamma_x + \frac{\gamma_x}{1-\gamma_x}\right) \hat{\w}\right).
\end{align*}
\section{Derivation of (\ref{forward_equation})}
\label{derivation_forward}
We recall that 
\begin{align*}
    \mathring{\v}(x,t) = \u(x,t) - \w(x+\gamma(x,t),t),
\end{align*}
set $\mathring{\phi}(x) = \phi(x+ \gamma(x,t))$ and use the notions (\ref{linearization}), (\ref{def_R_1}) and (\ref{def_R_2}). Furthermore, we write
\begin{align*}
    \mathcal{D}(\u) = \mathcal{J} \left(\begin{pmatrix}-\beta & 0 \\
    0 & -\beta \end{pmatrix}\u_{xx} + \begin{pmatrix}-\alpha & 0 \\
    0 & -\alpha \end{pmatrix}\u \right) - \u .
\end{align*}
Using that $\u$ and $\w$ are a solutions of (\ref{LLE_real}), we derive
\begin{align*}
    &(\partial_t - \mathcal{L}_0(\mathring{\phi})) \mathring{\v}(t)\\
    & = (\partial_t - \mathcal{D})(\u(t)) - (\partial_t - \mathcal{D})(\w)(\cdot+\gamma(\cdot,t),t) - \mathcal{N}'(\mathring{\phi})\mathring{\v}(t) + \mathcal{R}_5(\w(t),\gamma(t),\gamma_t(t)) \\&= \mathcal{N}(\mathring{\v}(t) + \tilde{\w}(\cdot+ \gamma(\cdot,t),t) + \mathring{\phi}) - \mathcal{N}(\tilde{\w}(\cdot+ \gamma(\cdot,t),t) + \mathring{\phi})  -\mathcal{N}'(\mathring{\phi})\mathring{\v}(t) + \mathcal{R}_5(\w(t),\gamma(t),\gamma_t(t)) \\
    &= \mathcal{R}_2(\mathring{\phi})(\tilde{\w}(\cdot+\gamma(\cdot,t),t),\mathring{\v}(t))+ \mathcal{R}_5(\tilde{\w}(t),\gamma(t),\gamma_t(t)) 
\end{align*}
with 
\begin{align*}
    &\mathcal{R}_5(\tilde{\w}(t),\gamma(t),\gamma_t(t)) \\
    &=  - \w_x(\cdot+\gamma(\cdot,t),t) \gamma_t(t) - \beta \mathcal{J}\left( \w_x(\cdot+\gamma(\cdot,t),t)(\gamma_{xx}(t))  + \w_{xx}(\cdot+\gamma(\cdot,t),t)(2\gamma_x(t) + \gamma_x(t)^2)\right)\\
    &= -\tilde{\w}_x(\cdot+\gamma(\cdot,t),t) \gamma_t(t) -\phi'(\cdot+\gamma(\cdot,t),t) \gamma_t(t) \\
    & \qquad - \beta\mathcal{J} \Bigl(\tilde{\w}_x(\cdot+\gamma(\cdot,t),t)(\gamma_{xx}(t))  + \tilde{\w}_{xx}(\cdot+\gamma(\cdot,t),t)(2\gamma_x(t) + \gamma_x(t)^2) \\
    &\qquad \qquad \quad\quad + \phi'(\cdot+\gamma(\cdot,t),t)(\gamma_{xx}(t))  + \phi''(\cdot+\gamma(\cdot,t),t)(2\gamma_x(t) + \gamma_x(t)^2) \Bigr).
\end{align*}
Recall that $\tilde{\w}(t) = \w(t)-\phi$ denotes the unmodulated perturbation of $\w$.

    \section{Proof of Proposition \ref{prop_existence_modulation functions}}
    \label{proof_local_mod}
    We recall the Duhamel formula (\ref{implicit_gamma}) given by
    \begin{align*}
    \gamma(t) = s_p(t)\v_0 + \int_0^t s_p(t-s)\Bigl(&\mathcal{R}_3(\hat{\w}(s),\hat{\v}(s),\gamma(s)) -\sigma_t(s)\hat{\v}_x(s) \\
    &+ (1-\gamma_x(s))\mathcal{R}_{2,2}(\hat{\w}(s),\hat{\v}(s)) + \mathcal{T}(\hat{\w}(s),\gamma(s)) \Bigr)\,ds
\end{align*}
whereas
\begin{align*}
  \hat{\v}(x,t) = \u(x- \sigma(t) - \gamma(x,t),t) - \hat{\w}(x,t) - \phi(x).
\end{align*}
To prevent confusion, we write
\begin{align*}
  \hat{\v}(t) = \hat{\v}(\gamma(t),t)
\end{align*}
and for the sake of readability, we introduce
\begin{align*}
    \tilde{\mathcal{N}}(t,\sigma(s),\gamma(s),s) = s_p(t-s)\Bigl(&\mathcal{R}_3(\hat{\w}(s),\hat{\v}(s),\gamma(s)) -\sigma_t(s)\hat{\v}_x(s) \\ &+ (1-\gamma_x(s))\mathcal{R}_{2,2}(\hat{\w}(s), \hat{\v}(s)) + \mathcal{T}(\hat{\w}(s),\gamma(s))\Bigr).
\end{align*}
We do a contraction argument. For this purpose, we establish nonlinear bounds.
\begin{lemma}
\label{contract_estimates_mod}
   Fix $0\leq \tau_1  \leq \tau_2<\min\{\tau_{\textrm{max}},t_{\textrm{max},\sigma}\}$. Let $j,k \in \mathbb{N}_0$ and fix $K>0$. Suppose that $$  \sup_{t \in [0,\tau_2]} \left(||\w(t)||_{W^{1,\infty}} + |\sigma_t(t)|+   ||\v(t)||_{W^{1,\infty}}  \right)\leq K.$$ There exist constants $C>0$ and $C_{j,k}>0$ such that we have the bounds
   \begin{align}
   \label{bound_v_hat_appendix}
       \sup_{s \in [\tau_1,\tau_2]}||\hat{\v}(\gamma_1(s),s)-\hat{\v}(\gamma_2(s),s)||_{L^2} \leq C\sup_{s \in [\tau_1,\tau_2]}||\gamma_1(s)- \gamma_2(s)||_{L^2},
   \end{align}
     and
        \begin{align}
        \begin{split}
            \sup_{t \in [\tau_1,\tau_2]} \sup\limits_{s \in [\tau_1, t]}&||\partial_t^j \partial_x^k (\tilde{\mathcal{N}}(t,\sigma(s),\gamma_2(s),s) - \tilde{\mathcal{N}}(t,\sigma(s),\gamma_2(s),s))||_{L^2} \\
             &\leq C_{j,k} \sup_{s \in [\tau_1,\tau_2]}(||\gamma_1(s)-\gamma_2(s)||_{H^2} + ||\partial_t\gamma_1(s) - \partial_t \gamma_2(s)||_{L^2}),
        \end{split}
             \label{bound_N_hat_appendix}
        \end{align}
    for any $\gamma_1,\gamma_2 \in C([\tau_1,\tau_2]; H^5(\R)) \times C^1([\tau_1,\tau_2]; H^3(\R))$ with $\sup_{t \in [\tau_1,\tau_2]}||\partial_x\gamma_1(t)||_{L^\infty}, ||\partial_x\gamma_1(t)||_{L^\infty} \leq \frac{1}{2}$.
\end{lemma}
\begin{proof}
Let $t \in [\tau_1,\tau_2]$ and $s\in [\tau_1,t]$. We rewrite
       \begin{align*}
           \hat{\v}(x,\gamma_1(s),s)-\hat{\v}(x,\gamma_2(s),s) &= \v(x- \sigma(s) - \gamma_1(x,s),s) -\v(x- \sigma(s) - \gamma_2(x,s),s) \\ & \quad   +  \w(x- \sigma(s) - \gamma_2(x,s),s) - \w(x- \sigma(s) - \gamma_1(x,s),s)  
       \end{align*}
       and since $\w(x- \sigma(s) - \gamma(x,s),s) - \phi(x) = \hat{\w}(x-\gamma(x,s),s) $, this yields (\ref{bound_v_hat_appendix}) by the mean value theorem.
      Let $j,k \in \mathbb{N}_0$.  Recalling the choice of $\mathcal{R}_3 = \mathcal{Q} + \partial_x \mathcal{S} + \partial_x^2 \mathcal{P}$ and the estimates on $s_p(t)$ from Proposition \ref{prop_linear_l_2}, we obtain
        \begin{align*}
            ||&\partial_t^j\partial_x^k s_p(t-s)(\mathcal{R}_3(\hat{\w}(s),\hat{\v}(\gamma_1(s),s),\gamma_1(s)) -\mathcal{R}_3(\hat{\w}(s),\hat{\v}(\gamma_2(s), s),\gamma_2(s)))||_{L^2} \\
            &\quad \leq C_{j,k} (||\gamma_1(s)-\gamma_2(s)||_{H^2} + ||\partial_t\gamma_1(s) - \partial_t \gamma_2(s)||_{L^2}).
        \end{align*}
        by taking derivatives on $s_p(t-s)$ and (\ref{bound_v_hat_appendix}).
Next, with the Cauchy-Schwarz inequality, we find
        \begin{align*}
            ||&\partial_t^j\partial_x^k s_p(t-s)\Bigl((1-\partial_x\gamma_1(s))\mathcal{R}_{2,2}(\hat{\w}(s),\hat{\v}(\gamma_1(s),s)) \\  & -(1-\partial_x\gamma_2(s))\mathcal{R}_{2,2}(\hat{\w}(s), \hat{\v}(\gamma_2(s),s))\Bigr)||_{L^2} \leq C_{j,k} ||\gamma_1(s)-\gamma_2(s)||_{H^1}.
        \end{align*}
        Moreover, we have
        \begin{align*}
              ||&\partial_t^j\partial_x^k s_p(t-s)\Bigl(\sigma_t(s)\left(\hat{\v}_x(\gamma_1(s),s) - \hat{\v}_x(\gamma_2(s),s)\right) \Bigr)||_{L^2} \leq C_{j,k} ||\gamma_1(s)-\gamma_2(s)||_{L^2}.
        \end{align*}
    Together with 
    \begin{align*}
      ||&\partial_t^j\partial_x^k s_p(t-s)\Bigl(\mathcal{T}(\hat{\w}(s), \gamma_1(s)) - \mathcal{T}(\hat{\w}(s),\gamma_1(s))\Bigr)||_{L^2}   \\
      &\leq C_{j,k} (||\gamma_1(s)-\gamma_2(s)||_{H^2} + ||\partial_t\gamma_1(s) - \partial_t \gamma_2(s)||_{L^2}),
    \end{align*}
   again taking derivatives on $s_p(t-s)$, we arrive at (\ref{bound_N_hat_appendix}).
\end{proof}
By the choice of $s_p$, we immediately have that $\gamma(t) = 0 $ for $t \in [0,1]$. We need to justify that we can extend the modulation function $\gamma$ to a maximal time such that the alternative (\ref{finite_time_cond_gamma}) holds. 
\begin{proof}[Proof of Proposition \ref{prop_existence_modulation functions}]
Let $\tilde{\gamma}$ be a solution of (\ref{implicit_gamma}) on $[0,t_0]$ with some $t_0>0$. Lemma \ref{contract_estimates_mod} tells us that 
\begin{align}
\begin{split}
        \gamma \mapsto  s_p(t)\v_0 + \int_{t_0}^t s_p(t-s)\Bigl(&\mathcal{R}_3(\hat{\w}(s),\hat{\v}(s),\gamma(s)) -\sigma_t(s)\hat{\v}_x(s) \\
    &+ (1-\gamma_x(s))\mathcal{R}_{2,2}(\hat{\w}(s),\hat{\v}(s)) + \mathcal{T}(\hat{\w}(s),\gamma(s)) \Bigr)\,ds
\end{split}
\label{contraction_mapping_for_gamma}
\end{align}
defines a contraction on \begin{align*}
    X_{t_0,\tau_0} = \Bigl\{\gamma \in & C([t_0,t_0+\tau_0]; H^5(\R)) \cap C^1([t_0,t_0+\tau_0]; H^3(\R)): \\
    &\sup_{s \in [t_0, t_0+ \tau_0]} ||(\gamma(s), \gamma_s(s))||_{H^5 \times H^3} < \frac{c}{2}\Bigr\}
\end{align*} for sufficiently small $\tau_0>0$. By the Banach fixed point theorem, there exists a unique solution $$\gamma \in C([t_0,t_0+\tau_0]; H^5(\R)) \cap C^1([t_0,t_0+\tau_0]; H^3(\R))$$
of 
\begin{align}
\begin{split}
\label{extend_gamma}
    \gamma(t) = s_p(t)\v_0 &+ \int_0^{t_0} s_p(t-s)\Bigl(\mathcal{R}_3(\hat{\w}(s),\hat{\v}(\tilde{\gamma}(s),s),\tilde{\gamma}(s)) -\sigma_t(s)\hat{\v}_x(\tilde{\gamma}(s),s) \\
    &\qquad \qquad \qquad\qquad+ (1-\gamma_x(s))\mathcal{R}_{2,2}(\hat{\w}(s),\hat{\v}(\tilde{\gamma}(s),s)) + \mathcal{T}(\hat{\w}(s),\tilde{\gamma}(s)) \Bigr)\,ds \\
    &+\int_{t_0}^t s_p(t-s)\Bigl(\mathcal{R}_3(\hat{\w}(s),\hat{\v}(\gamma(s),s),\gamma(s)) -\sigma_t(s)\hat{\v}_x(\gamma(s),s) \\
    & \qquad \qquad \qquad \qquad  +(1-\gamma_x(s))\mathcal{R}_{2,2}(\hat{\w}(s),\hat{\v}(\gamma(s),s)) + \mathcal{T}(\hat{\w}(s),\gamma(s)) \Bigr)\,ds
    \end{split}
\end{align}
for $t\in [t_0,t_0+\tau_0]$. Since (\ref{contraction_mapping_for_gamma}) defines a contraction on $X_{t_1,t_2}$ for every $t_1 \in (0, \min\{\tau_{\textrm{max}},t_{\textrm{max},\sigma}\})$ and sufficiently small $t_2>0$, we conclude that
\begin{align}
\label{gamma_taking_piece_together}
    \gamma(t) = \begin{cases}
        \tilde{\gamma}(t), & t\in [0,t_0]\\
        \gamma(t), & t\in [t_0,t_0+\tau_0]
    \end{cases} 
\end{align}
is the unique solution of (\ref{implicit_gamma}) on $[0,t_0 + \tau_0]$ such that $\gamma \in X_{t_0,\tau_0}$. Now, there exists a maximal time $t_{\textrm{max},\gamma} \in (0,\min\{\tau_{\textrm{max}},t_{\textrm{max},\sigma}\}]$ such that (\ref{max_gamma_cond}) holds. Assume that $t_{\textrm{max},\gamma}<\min\{\tau_{\textrm{max}},t_{\textrm{max},\sigma}\}$. If (\ref{finite_time_cond_gamma}) fails, then there exists a unique solution $\tilde{\gamma}$ of (\ref{implicit_gamma}) with $\sup_{t \in [0,t_{\textrm{max},\gamma})} ||(\tilde{\gamma}(t),\tilde{\gamma}_t(t))||_{H^5\times H^3} < \frac{c}{2}.$ This again allows to solve (\ref{extend_gamma}) on $X_{t_{\textrm{max},\gamma}, \tau_0'}$ for some small $\tau_0'$ giving a solution of (\ref{implicit_gamma}) via \ref{gamma_taking_piece_together} on $[0,t_{\textrm{max},\gamma} + \tau_0')$ with $ ||(\gamma(s), \gamma_s(s))||_{H^5 \times H^3} < \frac{c}{2}$ for all $t \in [0,t_{\textrm{max},\gamma} + \tau_0')$. This contradicts the maximality of $t_{\textrm{max},\gamma}$ and we conclude (\ref{finite_time_cond_gamma}).
\end{proof}

\bibliographystyle{abbrv}
\bibliography{refs}
\end{document}